\documentclass[12pt]{amsart}
\usepackage{graphicx}
\usepackage{amssymb}
\usepackage{tikz}

\usepackage{dsfont}
\allowdisplaybreaks

\newcommand{\ab}{\allowbreak}
\newcommand{\alg}{\textrm{alg}}
\newcommand{\E}{\textrm{E}}
\newcommand{\goe}{{\sc goe}}
\newcommand{\im}{\mathrm{Im}}
\newcommand{\la}{\langle}
\let\phi=\varphi
\newcommand{\ra}{\rangle}
\newcommand{\sa}{\sphericalangle}
\newcommand{\str}{t}
\newcommand{\tr}{\textrm{tr}}
\newcommand{\Tr}{\textrm{Tr}}

\newcommand{\bC}{\mathbb{C}}

\newcommand{\bR}{\mathbb{R}}
\newcommand{\bZ}{\mathbb{Z}}

\newcommand{\cA}{\mathcal A}

\newcommand{\cN}{\mathcal N}
\newcommand{\cP}{\mathcal P}
\newcommand{\cX}{\mathcal X}

\makeatletter
\newcommand\iraggedright{%
  \let\\\@centercr\@rightskip\@flushglue \rightskip\@rightskip
  \leftskip\z@skip}
\makeatother

\newcommand{\ds}{\displaystyle}

\newcommand{\thebottomline}{\renewcommand{\thefootnote}{}
  \renewcommand{\footnoterule}{}
  \phantom{M}\footnotetext{\tiny{}\hfill
    \textit{\noindent\romannumeral\day.%
\romannumeral\month.\romannumeral\year}}}

\newcommand\blfootnote[1]{%
  \begingroup
  \renewcommand\thefootnote{}\footnote{#1}%
  \addtocounter{footnote}{-1}%
  \endgroup
}

\theoremstyle{definition}
\newtheorem{theorem}{Theorem}
\newtheorem{definition}[theorem]{Definition}  
\newtheorem{lemma}[theorem]{Lemma}
\newtheorem{notation}[theorem]{Notation}
\newtheorem{remark}[theorem]{Remark}
\newtheorem{corollary}[theorem]{Corollary}

\newtheorem{example}[theorem]{Example}
\newtheorem{proposition}[theorem]{Proposition}

\title[Infinitesimal Distribution of the GOE]%
{non-crossing annular pairings and\\
The Infinitesimal Distribution of the GOE%
}

\author[j. a. mingo]{James A. Mingo} \address{Department
  of Mathematics and Statistics, Queen's University, Jeffery
  Hall, Kingston, Ontario, K7L 3N6, Canada}

\email{mingo@mast.queensu.ca} 

\thanks{Research supported by a Discovery Grant from the
  Natural Sciences and Engineering Research Council of
  Canada}

\begin{document}

\begin{abstract}
We present a combinatorial approach to the infinitesimal
distribution of the Gaussian orthogonal ensemble (\goe). In
particular we show how the infinitesimal moments are
described by non-crossing pairings, but not those of type
$B$. We demonstrate the asymptotic infinitesimal freeness of
independent complex Wishart matrices and compute their
infinitesimal cumulants. Using our combinatorial picture we
compute the infinitesimal cumulants of the \goe{} and
demonstrate the lack of asymptotic infinitesimal freeness of
independent Gaussian orthogonal ensembles. 
\end{abstract}

\maketitle

\section{Introduction}
\label{section:introduction}

Free independence was introduced by Dan Voiculescu in 1983
and since then there have been many extensions and
variations.  The common property of all these extensions is
that the mixed moments of independent random variables can
be computed by a universal rule from individual moments. The
rule depends on the type on independence being
considered. In this article we consider the infinitesimal
freeness of Belinschi and Shlyakhtenko
\cite{bs}. Infinitesimal probability spaces have recently
been used by Shlyakhtenko \cite{s} to understand small scale
perturbations in some random matrix models. Let us recall
some of the connections between free probability and random
matrix theory.\blfootnote{ AMS 2010 Mathematics Subject Classification: 46L54 (05D40 15B52 60B20)}

Let $\{A_N \}_N$ and $\{B_N\}_N$ be two self-adjoint
ensembles of random matrices. By this we mean that for each
integer $N \geq 1$ we have two self-adjoint matrices with
random entries. The eigenvalues of $A_N$, $\lambda^{(A)}_1
\leq \cdots \leq \lambda^{(A)}_N$, are thus random and we
form a random probability measure $\mu_N^{(A)}$ with a mass
of $1/N$ at each eigenvalue $\lambda^{(A)}_i$. We do the
same for $B_N$ and obtain another random measure
$\mu^{(B)}_N$. For many ensembles the random measures
$\mu^{(A)}_N$ and $\mu^{(B)}_N$ converge to deterministic
measures, called the limit eigenvalue distributions. Two
well known examples are Wigner's semi-circle law and the
Marchenko-Pastur law.

A central problem in random matrix theory is to compute the
limit eigenvalue distribution of $C_N = f(A_N, B_N)$ when
$f$ is a polynomial or a rational function in non-commuting
variables.  This would not be possible without some
assumptions on the `relative position' of $A_N$ and
$B_N$. By relative position we mean Voiculescu's notion of
freeness or one of its extensions. We do not need freeness
for finite $N$, but only in the large $N$ limit; when this
holds we say the ensembles are asymptotically free. When we
know that $A_N$ and $B_N$ are asymptotically free then we
can apply the analytic techniques of free probability
i.e. the $R$ and $S$ transforms (see \cite{vdn}) to compute
the limit distribution of $C_N$.

The first example of asymptotic freeness was given by
Voiculescu \cite{voi} where he showed that independent
self-adjoint Gaussian matrices were asymptotically
free. Since then there have been many generalizations and
elaborations.

Infinitesimal freeness is the branch of free probability
that enables us to model infinitesimal perturbations in the
same way as Voiculescu's theory did for $f(A_N, B_N)$. If we
start with $A_N$ as above but now assume that $B_N$ is a
non-random fixed finite rank self-adjoint matrix, recent
work of Shlyakhtenko \cite{s} and Belinschi and Shlyakhtenko
\cite{bs} shows that when $A_N$ is complex and Gaussian then
there is a universal rule for computing the effect on the
outlying eigenvalues. See Definition \ref{def:inf_freeness}
for a detailed definition. 


An infinitesimal distribution can be considered at the algebraic level or at the analytical level.
On the algebraic level an infinitesimal distribution is a pair
$(\mu, \mu')$ of linear functionals on $\bC[x]$ such that
$\mu(1) = 1$ and $\mu'(1) = 0$.  There are a few ways to
arrive at such a pair; we shall consider the ones arising
from random matrix models. Suppose $\{X_N\}_N$ is an
ensemble of self-adjoint random matrices where $X_N$ is $N
\times N$ and for all $k$ we have that the limit $\mu(x^k)
:= \lim_N \E(\tr(X_N^k))$ exists. Then the ensemble
$\{X_N\}_N$ has a limit distribution.  Suppose further that
for all $k$ we have $\mu'(x^k) := \lim_N N( \E(\tr(X_N^k)) -
\mu(x^k))$ exists. Then we say that the ensemble has a
\textit{infinitesimal distribution}. This was the context of
\cite{s}.

On the analytical level one can consider a pair $(\mu, \mu')$ 
of Borel measures on $\bR$ with $\mu$ being a probability measure
and $\mu'$ a signed measure with $\mu'(\bR) = 0$. 
An early example of an infinitesimal distribution
was that of the Gaussian orthogonal ensemble, given by
Johansson in \cite{kj}, also discussed by I.~Dumitriu and
A.~Edelman in \cite{de}, and Ledoux in \cite{l}. In this
case $\mu$ is Wigner's semi-circle law
\[
d\mu(x) = \frac{\sqrt{4 - x^2}}{2 \pi}\, dx \textrm{\ on\ }
[-2,2]
\]
and $\mu'$ is the difference of the Bernoulli and the
arcsine law:

\begin{equation}\label{eq:infinitesimal_goe}
d\mu'(x) = \frac{1}{2} \Big(\frac{\delta_{-2} + \delta_2}{2}
- \frac{1}{\pi}\frac{1}{\sqrt{4 - x^2}}\, dx\Big)
\textrm{\ on\ } [-2,2].
\end{equation}
Infinitesimal freeness was built on work of Biane, Goodman,
and Nica \cite{bgn} on freeness of type $B$.  While this
does provide a combinatorial basis for infinitesimal
freeness, we show in Theorem \ref{thm:inf_moments_diagrams}
that in the orthogonal, or `real' case, one needs to use the
annular diagrams of \cite{mn}. Since there is an additional
symmetry requirement (see the caption to
Fig. \ref{fig:into_fig}), we only need the outer half of the
diagram. This places infinitesimal freeness somewhere
between freeness and second order freeness.

Another example of an infinitesimal distribution was given
by Mingo and Nica in \cite[Corollary 9.4]{mn}, although is
was not then described as such because the infinitesimal
terminology didn't exist at the time. In \cite{mn} complex
Wishart matrices were considered. In particular $X_N =
\frac{1}{N}G^*G$ with $G$ a $M \times N$ Gaussian random
matrix with independent $\cN(0,1)$ entries.  When $\lim_N
M/N = c$ we get the well known Marchenko-Pastur
distribution with parameter $c$ (see
\cite[Ex. 2.11]{ms2017}). If we further assume that $c' :=
\lim_N (M - Nc)$ exists then there is an infinitesimal
distribution with $\mu'$ given by
\begin{equation}\label{eq:inf_sing_mom_cx_wish}
d\mu'(x) = -c'
\begin{cases}
\phantom{\frac{1}{2}}\delta_0 -  \frac{x + 1 -c}
{2 \strut\pi x \sqrt{(b - x)(x - a)}}\,dx & c < 1 \\
\frac{1}{2} \delta_0 -  \frac{1}{2\pi\sqrt{x(4 - x)}}\,dx & c = 1 \\
\phantom{\frac{1}{2} \delta_0}
-\frac{x + 1 -c}
{2 \strut\pi x \sqrt{(b - x)(x - a)}}\,dx  & c > 1 \\
\end{cases}.
\end{equation}
Note that the continuous part of $\mu'$ is supported on the
interval $[a, b]$ with $a = (1 - \sqrt{c})^2$ and $b = (1 +
\sqrt{c})^2$. In Remark \ref{rem:formal_derivative} we show that at a formal level we can consider $\mu'$ to be a derivative of $\mu$. However, in \cite{mn} the distribution was given in
terms of infinitesimal cumulants: $\kappa'_n = c'$ for all
$n$, where $\kappa_n'$ is an infinitesimal cumulant; the
density above is obtained from Equation
(\ref{eq:infinitesimal_r-transform}) below. The intuitive
idea is to regard $c'$ as the derivative, as $1/N
\rightarrow 0$, of the shape parameter $c$. For a very
simple case, take $c = 1$ and $c' \in \bZ$ an integer. We
let $M = N + c'$, then $M/N \rightarrow c$ and $M - cN =
c'$. Earlier authors only considered the case $c' = 0$,
which one can always arrange by taking $(M_k, N_k)$ to be
the $k^{\mathrm{th}}$ convergent in the continued fraction
expansion of $c$. 

Note that in Bai and Silverstein \cite{bai_s} and in the
work of many other authors, especially in statistics, a
different normalization is used for Wishart matrices. This
produces a slightly different limit distribution, related to
one used here by a simple change of variable.  See
\cite[Remark 2.12]{ms2017}.  So all of our results can be easily
transferred to the other normalization. Whenever clarity
permits we shall omit the dependence of our matrices on $N$,
thus in the expressions above we wrote $X$ instead of $X_N$.

\setbox1=\hbox{\includegraphics[scale=0.5]{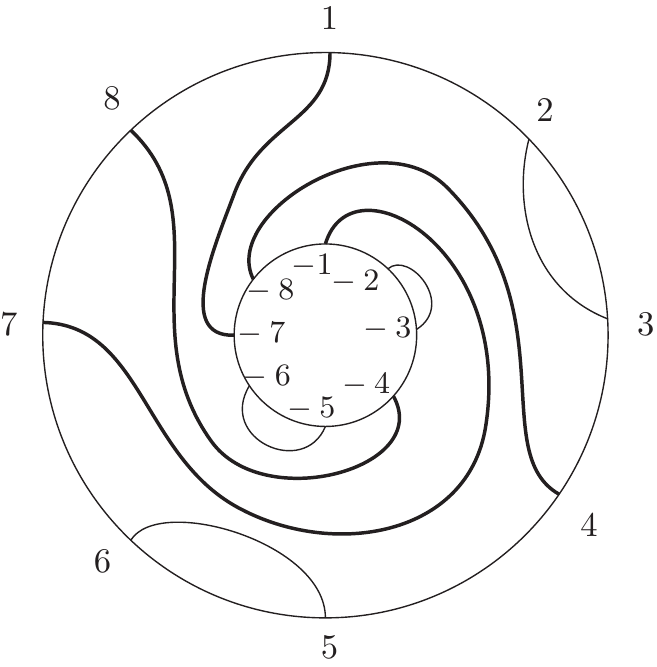}}

\begin{figure}\label{fig:into_fig}
$\vcenter{\hsize=\wd1\box1}$\qquad $\vcenter{\hsize =
    200pt\raggedright\small {\sc Figure \ref{fig:into_fig}} 
    \refstepcounter{figure}
    The planar objects are the
    non-crossing annular pairings of
    \cite{mn}, except in this case the circles have the same
    orientation. Moreover we require that $(r, -r)$ is never
    a pair and if $(r, s)$ is a pair then $(-r, -s)$ is also
    a pair. These are the only conditions.}$
\end{figure}

The $\frac{1}{N}$ expansion of $\E(\tr(X_N^n))$ in the
\goe{} case is known to count maps on locally orientable
surfaces (see \cite[Thm. 1.1]{gj} and \cite[\S 5]{l}). What
is new in this article is that the infinitesimal moments of
the \goe{} are described by planar objects and thus stay
within the class of the non-crossing partitions standard in
free probability, but not the non-crossing partitions of
type $B$ used in \cite{bgn}. We shall also see that independent
\goe{}'s are not asymptotically infinitesimally free, nor
are a \goe{} and a deterministic matrix. However there is a
universal rule for computing mixed moments (see Theorem
\ref{thm:lack_inf_free}).

Another new point in our presentation is the simple
relation: $g(z) = - r(G(z)) G'(z)$ between the infinitesimal
Cauchy transform $g$ and the infinitesimal
$r$-transform. This simplifies a number of our computations.

In \S \ref{section:infinitesimal_freeness} we present of
review of infinitesimal freeness and infinitesimal
cumulants. In \S
\ref{section:infinitesimal_moments_of_a_goe_random_matrix}
we find the combinatorial expression for the infinitesimal
moments. In \S \ref{section:inf_moments} we present the main combinatorial object of this paper, $NC_2^\delta(n, -n)$, as illustrated in Figure \ref{fig:into_fig}.
We show how the infinitesimal moments of the \goe{}
are described by these non-crossing partitions. In \S
\ref{section:inf_cum_goe} we use this description to find
the infinitesimal cumulants of the \goe{} and then show that
independent \goe{} matrices are not asymptotically
infinitesimally free. In \S \ref{section:asymptotic} we show
how the results of \cite[\S 9]{mn} give the infinitesimal
cumulants of a complex Wishart matrix and demonstrate
asymptotic infinitesimal freeness.  In \S
\ref{section:universal} we show that a \goe{} ensemble and
constant matrices are not asymptotically infinitesimally
free but do satisfy a universal law. This demonstrates the
difference between the complex and real case.

\section{Infinitesimal freeness}
\label{section:infinitesimal_freeness}
The theory of infinitesimal freeness and infinitesimal
cumulants is presented in \cite{bs}, \cite{bgn}, and
\cite{fn}. See also \cite{f}. We shall extract the parts
needed for our results.

We begin by recalling the moment-cumulant formula
(\cite[Lect. 11]{ns}). For a non-commutative probability
space $(\cA, \phi)$ and $a \in \cA$ we let $m_n = \phi(a^n)$
and call $\{m_n \}_n$ the \textit{moment sequence} of
$a$. Let us recall the usual way of constructing the free
cumulants $\{\kappa_n \}_n$. Suppose we have for each $n$ a
linear map $\kappa_n: \cA^{\otimes n} \rightarrow \bC$. We
can extend this to a sequence of maps indexed by partitions
by setting for $\pi \in \cP(n)$
\[
\kappa_\pi (a_1, \dots, a_n) =
\mathop{\prod_{V \in \pi}}_{V = (i_1, \dots , i_l)}
\kappa_l(a_{i_1}, \dots, a_{i_l}). 
\]
We then in turn use this to define $\{\kappa_n\}_n$ by the
relations
\begin{equation}\label{eq:moment-cumulant}
\phi(a_1 \cdots a_n) =
\sum_{\pi \in NC(n)} \kappa_\pi(a_1, \dots, a_n).
\end{equation}
This produces an inductive and recursive definition because
on the right hand side of (\ref{eq:moment-cumulant}) there
is only one term with a $\kappa_n$ and for all the others we
only need to know $\kappa_1, \dots, \kappa_{n-1}$.

Now let us recall the definition of an infinitesimal
probability space \cite{bs}. We start with a non-commutative
probability space $(\cA, \phi)$ and suppose we have $\phi':
\cA \rightarrow \bC$ with $\phi'(1) = 0$. We use the
infinitesimal version of (\ref{eq:moment-cumulant}) to
define the infinitesimal cumulants:
\begin{equation}\label{eq:inf_moment_cumulant}
\phi'(a_1 \cdots a_n) =
\sum_{\pi \in NC(n)} \partial\kappa_\pi(a_1, \dots, a_n)
\end{equation}
where the maps $\partial\kappa_\pi: \cA^{\otimes n} \rightarrow \bC$
are defined as follows.

Given a sequence of pairs $(\kappa_n, \kappa'_n)$ of linear maps $\kappa_n, \kappa'_n: \cA^{\otimes n}
\rightarrow \bC$ we define $\kappa'_{(\pi, V)}$ where
$\pi \in \cP(n)$ and $V \in \pi$ as follows. If $V = (i_1,
\dots, i_l)$ we set
\[ 
\kappa'_{\pi, V}(a_1, \dots, a_n) =
\kappa'_l(a_{i_1}, \dots, a_{i_l}) \kern-1em
\mathop{\mathop{\prod_{W \in \pi}}_{W \not = V}}_{W = (j_1,
  \dots, j_m)} \kern-1em \kappa_m(a_{j_1}, \dots, a_{j_m}).
\]
and
\begin{equation}\label{eq:diff_cumulant}
\partial\kappa_\pi(a_1, \dots, a_n)  =
\sum_{V \in \pi} \kappa'_{\pi, V}(a_1, \dots, a_n). 
\end{equation}
So given $(\phi, \phi')$ we produce a well defined sequence $\{ \kappa_n, \kappa'_n \}_n$ from  (\ref{eq:moment-cumulant}) 
and (\ref{eq:diff_cumulant}) as we did for the
free cumulants $\{\kappa_n\}_n$. We round out the notation by setting 
$\partial\kappa_n = \kappa'_n$. 

\begin{example}
Suppose we have an infinitesimal distribution such that
$\kappa_n = c$ for all $n$ and $\kappa'_n = c'$ for all
$n$. We are assuming that $c$ and $c'$ are real
numbers. Then $\partial \kappa_\pi = c'\cdot \#(\pi)\cdot
c^{\#(\pi) - 1}$, as for each $V \in \pi$ we have
$\kappa'_{\pi, V} = c' \cdot c^{\#(\pi) - 1}$ and there are
$\#(\pi)$ blocks $V$.
\end{example}

For use in \S \ref{section:universal}, we apply the
$\partial$ notation to $\phi$ by setting
\[
\partial \phi_\pi(a_1, \dots, a_n)
=
\sum_{V \in \pi} \phi_{\pi, V}(a_1, \dots, a_n),
\]
where, when $V = (i_1, \dots, i_k)$ we have 
\[
\phi_{\pi, V}(a_1, \dots, a_n)
=
\phi'(a_{i_1} \cdots a_{i_k})
\mathop{\prod_{W \not= V}}_{W = (j_1, \dots, j_l)}
\phi(a_{j_1} \cdots a_{j_l}).
\]
In this notation 
\begin{equation}\label{eq:infinitesimal_cumulant_moment}
\partial\kappa_\pi(a_1, \dots, a_n)
=
\mathop{\sum_{\sigma \in NC(n)}}_{\sigma \leq \pi} \mu(\sigma, \pi)
\partial\phi_\pi(a_1, \dots, a_n)
\end{equation}

We shall clarify these relations by looking at the cases $n =
1, 2$ and $3$.  For $n = 1$ we have
\[
\phi'(a_1) = \kappa'_1(a_1).
\]
So $\kappa'_1(a_1) = \phi'(a_1)$. For $n =2 $ we have 
\[
\phi'(a_1, a_2) = \kappa'_2(a_1, a_2) + \kappa'_1(a_1)
\kappa_1(a_2) + \kappa_1(a_1) \kappa'_1(a_2).
\]
Thus $\kappa'_2(a_1, a_2) = \phi'(a_1a_2) - \{\phi'(a_1)
\phi(a_2) + \phi(a_1) \phi'(a_2)\}   $. For $n = 3$ we have
\begin{eqnarray*}\lefteqn{
\phi'(a_1a_2a_3) = \kappa'_3(a_1, a_2, a_3) + \kappa'_1(a_1)
\kappa_2(a_2, a_3) + \kappa_1(a_1) \kappa'_2(a_2, a_3)
  }\\ && \mbox{} + \kappa'_1(a_2) \kappa_2(a_1, a_3) +
  \kappa_1(a_2) \kappa'_2(a_1, a_3)+ \kappa'_1(a_3)
  \kappa_2(a_1, a_2) \\ && \mbox{} + \kappa_1(a_3)
  \kappa'_2(a_1, a_2) + \kappa'_1(a_1) \kappa_1(a_2)
  \kappa_1(a_3) + \kappa_1(a_1) \kappa'_1(a_2)
  \kappa_1(a_3)\\ && \mbox{} + \kappa_1(a_1) \kappa_1(a_2)
  \kappa'_1(a_3).
\end{eqnarray*}
From which we conclude that
{\setlength\arraycolsep{2pt}
\begin{eqnarray*}\lefteqn{
\kappa'_3(a_1, a_2, a_3) = \phi'(a_1a_2a_3)} \hspace{-1em}\\
&& \mbox{}
-\{ 
\phi'(a_1) \phi(a_2a_3) + \phi(a_1) \phi'(a_2a_3) \\
&& \mbox{}
+ \phantom{\{}
\phi'(a_2) \phi(a_1a_3) + \phi(a_2) \phi'(a_1a_3) \\
&& \mbox{}
+ \phantom{\{}
\phi'(a_3) \phi(a_1a_2) + \phi(a_3) \phi'(a_1a_2) \}\\
\mbox{}
&& \mbox{}  +
2 \{ \phi'(a_1) \phi(a_2) \phi(a_3) + \phi(a_1) \phi'(a_2)
\phi(a_3) + \phi(a_1) \phi(a_2) \phi'(a_3) \}. \\
\end{eqnarray*}}
These examples are special cases of the M\"obius inversion
of Eq.~(\ref{eq:inf_moment_cumulant})
\[
\kappa'_n(a_1, \dots, a_n) =
\sum_{\pi \in NC(n)} \mu(\pi, 1_n) \ 
\partial\phi_\pi(a_1, \dots a_n).
\]
When all the random variables are the same we can just write
everything in terms of $\{m_n , m'_n\}_n$ and $\{\kappa_n,
\kappa'_n\}_n$. If $\pi$ has blocks of size $k_1, k_2,
\dots, k_l$ we can write, using the notation of equation
(\ref{eq:diff_cumulant}),
\[
\partial m_\pi = \sum_{p=1}^l m_{k_1} \cdots m_{k_{p-1}} m'_{k_p}
m_{k_{p+1}} \cdots m_{k_l}
\]
which is the Leibniz rule applied to 
\[
m_\pi = \prod_{p=1}^n m_{k_p}. 
\]
Recall that the Cauchy transform of $\mu$ is given by
\[
G(z) = \sum_{n=0}^\infty \frac{m_n}{z^{n+1}}
= \int_{\bR} (z - t)^{-1}\, d\mu(t).
\] 
If the corresponding cumulants are $\{\kappa_n\}_n$ then the
$R$-transform is
\[
R(z) = \kappa_1 + \kappa_2 z + \kappa_3 z^3 + \cdots.
\]
The Cauchy transform and the $R$-transform are related by
the Voicu\-lescu equations
\begin{equation}\label{eq:voiculescu_relation}
\frac{1}{G(z)} + R(G(z)) = z = G\Big(\frac{1}{z} + R(z)\Big).
\end{equation}
In the infinitesimal case we proceed as in
\cite[Thm. 6]{bs}. For $z, w \in \bC$ we let $Z$ be the
matrix
\[
Z = \begin{pmatrix}z & w \\ 0 & z\end{pmatrix}. 
\]
Then
\[ 
Z^{n} = \begin{pmatrix}z^n & nz^{n-1} w \\ 0 & z^n\end{pmatrix} 
\mathrm{\ and\ } 
Z^{-n} = \begin{pmatrix}z^{-n} & -nz^{-(n+1)} w \\ 0 & z^{-n}\end{pmatrix}.
\]
To create the infinitesimal Cauchy and $R$-transform we set
\[
M_n = \begin{pmatrix} m_n & m'_n \\ 0 & m_n \end{pmatrix}
\mathrm{\ and\ }
K_n = 
\begin{pmatrix} \kappa_n & \kappa'_n \\ 0 & \kappa_n \end{pmatrix}.
\]
Then we let
\begin{eqnarray*}
G(Z) &=& \sum_{n=0}^\infty M_n Z^{-(n+1)} \\
& = &
\sum_{n=0}^\infty 
\begin{pmatrix} \ds\frac{m_n}{z^{n+1}} & \ds\frac{-(n+1) m_n}{ z^{n+2}}w  + \frac{m'_n}{z^{n+1}} \\ 
               0 & \ds\frac{m_n}{z^{n+1}} \end{pmatrix} \\
& = & 
\begin{pmatrix} \ds\sum_{n=0}^\infty \frac{m_n}{z^{n+1}} 
& \ds w\sum_{n=0}^\infty \frac{-(n+1) m_n}{z^{n+2}} 
+ \ds\sum_{n=0}^\infty \frac{m'_n}{z^{n+1}} \\ 
0 & \ds\sum_{n=0}^\infty \frac{m_n}{z^{n+1}} \end{pmatrix} \\
& = & 
\begin{pmatrix} G(z) & G'(z) w + g(z) \\ 0 & G(z) \end{pmatrix}.              
\end{eqnarray*}
Here $g$ is  the \textit{infinitesimal Cauchy transform}
\[
g(z) = \frac{m'_1}{z^2} + \frac{m'_2}{z^3} +
\cdots = \int_{\bR} (z - t)^{-1}\, d\mu'(t)
\]
and $G' = \ds\frac{dG}{dz}$. Likewise we set 
\begin{eqnarray*}
R(Z) &=& \sum_{n=1}^\infty K_n Z^{n-1} \\
& = &
\sum_{n=1}^\infty 
\begin{pmatrix} \kappa_n & \kappa'_n \\ 0 & \kappa_n \end{pmatrix}
\begin{pmatrix}z^{n-1} & (n-1)z^{n-2} w \\ 0 & z^{n-1}\end{pmatrix} \\
& = &
\sum_{n=1}^\infty 
\begin{pmatrix} \kappa_n z^{n-1} & 
(n-1)\kappa_n z^{n-2} w + \kappa'_n z^{n-1}\\ 0 & \kappa_n z^{n-1}\end{pmatrix} \\
& = & 
\begin{pmatrix} R(z) & R'(z) w + r(z) \\ 0 & R(z) \end{pmatrix}
\end{eqnarray*}
where $r$ the infinitesimal $r$-transform
\[
r(z) = \kappa'_1 + \kappa'_2 z + \kappa'_3 z^2 + \cdots
\]
and $R' = \ds \frac{dR}{dz}$. The infinitesimal versions of
the Voiculescu equations (\ref{eq:voiculescu_relation}) are
\[
\big(G(Z)\big)^{-1} + R(G(Z)) = Z = R(Z^{-1} + G(Z)).
\]
Let us use this to find the relation between $r$ and $g$,
the infinitesimal versions of $R$ and $G$.  First
\[ \big(G(Z)\big)^{-1}
= \begin{pmatrix}
G(z)^{-1} & -G(z)^{-2}[ w G'(z) + g(z)] \\ 0 & G(z)^{-1} \\
\end{pmatrix}. 
 \]
Next
\[
R(G(Z)) =
\begin{pmatrix}
R(G(z)) & [wG'(z) + g(z)] R'(G(z)) + r(G(z)) \\
0 & R(G(z)) \\
\end{pmatrix}.
\]
Thus
\begin{eqnarray*}\lefteqn{
\begin{pmatrix}z&w\\0&z\end{pmatrix}= Z = \big( G(Z)\big)^{-1} + R(G(Z)) }\\
& = &
\begin{pmatrix}
G(z) ^{-1} + R(G(z)) &
[wG'(z) + g(z)] [-G(z)^{-2} + R'(G(z))]  \\
  & + r(G(z)) \\
0 & G(z) ^{-1} + R(G(z)) \\
\end{pmatrix}.
\end{eqnarray*}
Hence
\begin{eqnarray*}
w &=& [wG'(z) + g(z)] [-G(z)^{-2} + R'(G(z))] + r(G(z)) \\
& = &
w G'(z) [-G(z)^{-2} + R'(G(z))] 
+ g(z) [-G(z)^{-2} + R'(G(z))] \\
&& \qquad\mbox{} + r(G(z)) \\
& = &
w + g(z) [G'(z)]^{-1} + r(G(z)),
\end{eqnarray*}
where we have used the derived Voiculescu relation
\[
G'(z) [-G(z)^{-2} + R'(G(z))] = 1.
\]
\begin{theorem}
The infinitesimal Cauchy and $r$-transforms are related by
the equations
\begin{equation}\label{eq:infinitesimal_r-transform}
g(z) = - r(G(z)) G'(z)
\end{equation}
and
\[
r(z) = -g\big(K(z)\big) K'(z)
\]
where $K(z) = \ds\frac{1}{z} + R(z) = G^{\langle-1\rangle}(z)$ and
$K' = \ds\frac{dK}{dz}$. 
\end{theorem}

\begin{remark}
Note that for the infinitesimal versions, $g$ and $r$, we
don't have to solve an equation to get one from the
other. This is one similarity with second order freeness
where the second order Cauchy and $R$-transforms are
related by
\begin{multline*}
G(z, w) = R(G(z), G(w)) G'(z) G'(w) \\
 \mbox
+ \frac{\partial^2}{\partial z \partial w} \log \Big(
\frac{1/G(z) - 1/G(w)}{z - w}\Big)
\mbox{\ (see \cite[Ch. 5]{ms2017} and \cite[Cor.~6.4]{cmss}). }
\end{multline*}
\end{remark}

Let us recall the notions of asymptotic freeness from
\cite{fn} that we shall use. First we shall give the
original definition and an equivalent formulation, which
will be what we actually use in this paper.

\begin{definition}\label{def:inf_freeness}
Let $(\cA, \phi, \phi')$ be an infinitesimal probability
space and $\cA_1, \dots, \cA_s$ be unital subalgebras. We
say that the subalgebras $\cA_1, \dots,\ab \cA_s$ are
\textit{infinitesimally free} if for all $a_1, \dots, a_n
\in \cA$ with $\phi(a_i) =0$ for $i = 1, \dots, n$ and $a_i
\in \cA_{j_i}$ with $j_1 \not = j_2 \not = \cdots \not =
j_{n - 1} \not = j_n$ we have
\begin{enumerate}

\item $\phi(a_1 \cdots a_n) = 0$

\item $\phi'(a_1 \cdots a_n) = 0$ for $n$ even and for $n =
  2m + 1$ odd we have
\[
\phi'(a_1 \cdots a_n)
= \phi(a_1a_n) \phi(a_2 a_{n-1}) \cdots \phi(a_ma_{m+2})
\phi'(a_{m+1}). 
\]
\end{enumerate}
\end{definition}

We extend this definition to individual random variables in
the usual way.

\begin{definition}
Let $(\cA, \phi, \phi')$ be an infinitesimal probability
space. Suppose we are given elements $x_1, \dots, x_s \in \cA$ and let
$\cA_i = \alg(1, x_i)$ be the algebra generated by $1$ and
$x_i$. We say that the elements $x_ 1, \dots , x_s$ are
\textit{infinitesimally free} if the subalgebras $\cA_1,
\dots, \cA_s$ are infinitesimally free.
\end{definition}

We shall obtain our asymptotic freeness results using the
characterization of infinitesimal freeness in terms of
cumulants. 

\begin{definition}
Let $(\cA, \phi, \phi')$ be an infinitesimal probability
space and $x_1, \dots, x_s \in \cA$. Suppose that for all
$n$-tuples $i_1, \dots, i_n \in [s]$ such that they are not
all equal we have both $\kappa_n(x_{i_1}, \dots, x_{i_n}) =
0$ and $\kappa'_n(x_{i_1}, \dots, x_{i_n}) = 0$. Then we say
\textit{mixed cumulants vanish}.
\end{definition}

\begin{theorem}[\cite{fn} Cor. 4.8]\label{thm:vanishing_mixed_cumulants}
Let $(\cA, \phi, \phi')$ be an infinitesimal probability
space and $\cX_1, \dots, \cX_s$ be subsets of $\cA$. Then
$\cX_1, \dots, \cX_s$ are infinitesimally free if and only
if mixed cumulants vanish.
\end{theorem}

\section{The infinitesimal moments of a goe random matrix}
\label{section:infinitesimal_moments_of_a_goe_random_matrix}
In this section we make precise the notation we shall use 
to describe \goe{} random matrices. 

\begin{notation}\label{notation:1}
Let $G = \frac{1}{\sqrt{N}} (g_{ij})_{ij}$ with $\{
g_{ij}\}_{ij}$ independent identically distributed $\cN(0,
1)$ random variables and $X = \frac{1}{\sqrt{2N}}(G +
G^\str)$.  Then $X$ is a $N \times N$ \goe{} random
matrix. 

\begin{remark}\label{rem:independent_sum}
Note that the linear combination $\frac{1}{\sqrt{2}}(X_N +
Y_N)$ of two independent \goe{} random matrices $\{X_N\}_N$
and $\{Y_N\}_N$ is again a \goe{} random matrix. Indeed if
$X = \frac{1}{\sqrt{2N}}( G_1 + G_1^\str)$ and $Y =
\frac{1}{\sqrt{2N}}( G_2 + G_2^\str)$, then let $G_3 =
\frac{1}{\sqrt{2}} (G_1 + G_2)$. Then the entries of $G_3$
are independent $\cN(0,1)$ random variables, so $Z =
\frac{1}{\sqrt{2N}}(G_3 + G_3^\str)$ is a \goe{} random
matrix.
\end{remark}

We will denote by $\tr$ the normalized trace of a $N
\times N$ matrix. Our goal in this section is to compute in
terms of planar diagrams the $1/N$ term of $\E(\tr(X^n))$
for each $n$. We shall see that for $n$ odd we have
$\E(\tr(X^n)) = 0$.

\begin{center}
\begin{tabular}{rl}
$n$ & $\E(\tr(X^n))$ \\ [2pt]\hline
2  & $1 + N^{-1}$ \vrule width 0pt height 12pt\\
4 & $2 + 5 N^{-1} + 5 N^{-2}$\\
6 & $5 + 22 N^{-1} + 52 N^{-2} + 41 N^{-3}$ \\
8 & $14 + 93N^{-1} + 374 N^{-2} + 690 N^{-3} + 509 N^{-4}$ \\
10 & $42 + 386 N^{-1} + 2290 N^{-2} + 7150 N^{-3} + 12143 N^{-4} + 8229 N^{-5}$
\end{tabular}
\end{center}
\noindent
The constant terms are the familiar Catalan numbers; the
coefficients of $N^{-1}$ are the moments of the $\mu'$ in
Eq. (\ref{eq:infinitesimal_goe}).

In this paper we shall frequently use the following
notation: for any matrix $A$ we set $A^{(-1)} = A^\str$ and
$A^{(1)} = A$.

Most of our calculations will be in $S_n$, the symmetric
group on $[n] = \{1, 2, 3, \dots, n\}$. Let $\gamma = (1, 2,
3, \dots, n) \in S_n$ be the permutation with one cycle.

We let $[\pm n] = \{ 1, 2,3, \dots, n, -n, -(n-1), \dots,
-1\}$ and $S_{\pm n}$ the permutations of $[\pm n]$. We
embed $S_n$ into $S_{\pm n}$ by making $\pi \in S_n$ act
trivially on $\{ -n, -(n-1), \dots, -1\}$. We let $\delta
\in S_{\pm n}$ be the permutation $\delta(k) = -k$ for all
$k \in [\pm n]$. For any permutation $\pi$ we let $\#(\pi)$
denote the number of cycles of $\pi$. Note that $\#(\pi
\sigma) = \#(\sigma \pi)$. If the subgroup $\langle \pi,
\sigma \rangle$ generated by $\pi$ and $\sigma$ acts
transitively on $[n]$ then there is an integer $ g \geq 0$
(the genus of a certain surface) such that
\begin{equation}\label{eq:euler_characteristic}
\#(\pi) + \#(\sigma\pi^{-1}) + \#(\sigma) = n + 2(1 - g).
\end{equation}
This is Euler's equation for the Euler characteristic of the
corresponding surface. Any permutation $\pi$ is
automatically considered a partition whose blocks are the
cycles of $\pi$. The partition will be non-crossing if and
only if
\[
\#(\pi) + \#(\pi^{-1}\gamma) = n + 1.
\]

We set $\bZ_2 = \{-1, 1\}$, if $\epsilon \in \bZ_2^n$ we
write $\epsilon = (\epsilon_1, \dots, \epsilon_n)$. We shall
also regard $\epsilon$ as a permutation in $S_{\pm n}$ as
follows. If $k \in [\pm n]$ we let $\epsilon(k) =
\epsilon_{|k|} k$. If $\epsilon = (-1, -1, \dots, -1)$ then
$\epsilon = \delta$. As permutations $\epsilon$ and $\delta$
commute.

A partition is a \textit{pairing} if all its blocks have 2
elements.  $\cP_2(n)$ is the set of pairings of $[n]$ (empty
of $n$ is odd).

Given $j: [\pm n] \rightarrow [N]$ we let $\ker(j)$ be the
partition of $[\pm n]$ such that $j$ is constant on the
blocks of $\ker(j)$ and takes on different values on
different blocks. If $\ker(j) \geq \gamma \delta
\gamma^{-1}$ then $j_{-1} = j_2$, $j_{-2} = j_3$, \dots,
$j_{-n} = j_1$. If $P$ is a true/false proposition depending on
a variable $x$ we write $\mathds{1}_P$ to be the function 
\[
\mathds{1}_P(x) =
\begin{cases}
1 & P(x) \mathrm{\ is\ true} \\
0 & P(x) \mathrm{\ is\ false}
\end{cases}.
\]
Thus for our Gaussian matrix $G = (g_{ij})_{ij}$ we have the Wick formula
\[
\E(g_{i_1i_{-1}} \cdots g_{i_ni_{-n}})
=
\sum_{\pi \in \cP_2(n)}
\mathds{1}_{\ker(i) \geq \pi\delta\pi\delta}.
\]
\end{notation}

\begin{lemma}\label{lemma:1}
\begin{equation}\label{equation:1}
\E(\tr(X^n)) =  \kern-0.5em
  \sum_{\pi \in \cP_2(n)} \sum_{\epsilon \in
  \bZ_2^n} 2^{-n/2} N^{ \#( \epsilon \gamma \delta \gamma^{-1}
  \epsilon \vee \pi \delta \pi \delta) -(n/2 + 1)}
\end{equation}
\end{lemma}

\begin{proof}
Let $i = j \circ \epsilon$. Then
$G^{(\epsilon_k)}_{j_kj_{-k}} = g_{i_ki_{-k}}$. So
\begin{eqnarray*}\lefteqn{
\E(\tr(X^n))} \\ 
& = & N^{-(n/2 + 1)}2^{-n/2}
  \sum_{\epsilon_1, \dots, \epsilon_n = \pm 1}
  \E(\Tr(G^{(\epsilon_1)} \cdots G^{(\epsilon_n)})) 
\\ & = &
  N^{-(n/2 + 1)}2^{-n/2} \kern-1em \sum_{\epsilon_1, \dots,
    \epsilon_n = \pm 1} \mathop{\sum_{j_{\pm 1}, \dots,
      j_{\pm n} = 1}^N}_{\ker(j) \geq
  \gamma\delta\gamma^{-1}} \E(G^{(\epsilon_1)}_{j_1j_{-1}}
    \cdots G^{(\epsilon_n)}_{j_nj_{-n}}) \\ 
& = & N^{-(n/2 +
      1)}2^{-n/2} \kern-1em \sum_{\epsilon_1, \dots,
      \epsilon_n = \pm 1} \mathop{\sum_{i_{\pm 1}, \dots,
        i_{\pm n} = 1}^N}_{\ker(i) \geq
    \epsilon\gamma\delta\gamma^{-1}\epsilon}
      \E(g_{i_1i_{-1}} \cdots g_{i_ni_{-n}} ).
\end{eqnarray*}

Now $\E(g_{i_1i_{-1}} \cdots g_{i_ni_{-n}} ) = \#(\{\pi \in
\cP_2(n) \mid i_r = i_s$ and $i_{-r} = i_{-s}$ whenever $(r,
s) \in \pi\})$. Thus

\[
\mathop{\sum_{i_{\pm 1}, \dots, i_{\pm n} = 1}^N}_%
{\ker(i)\geq \epsilon\gamma\delta\gamma^{-1}\epsilon}
  \E(g_{i_1i_{-1}} \cdots g_{i_ni_{-n}} ) = \sum_{\pi \in
    \cP_2(n)} N^{ \#( \epsilon \gamma \delta \gamma^{-1}
    \epsilon \vee \pi \delta \pi \delta)}.
\]
\end{proof}

\begin{remark}\label{remark:pair_decomposition}
We have to decide for a pair $(\pi, \epsilon)$ what the
value of $\#( \epsilon \gamma \delta \gamma^{-1} \epsilon
\vee \pi \delta \pi \delta) -(n/2 + 1)$ can be. Recall that
if $p$ and $q$ are pairings and $p \vee q$ denotes the join
as partitions then $2\#(p \vee q) = \#(pq)$ (see \cite[Lemma
  2]{mp}). Moreover we can write the cycle decomposition of
$pq$ as $pq = c_1c_1' \cdots c_k c_k'$ where $c_l' = q
c_l^{-1} q$.
\end{remark}

\begin{lemma}
For $\pi \in \cP_2(n)$ and $\epsilon \in\bZ_2^n$ we have
$\#( \epsilon \gamma \delta \gamma^{-1} \epsilon \vee \pi
\delta \pi \delta) \ab -(n/2 + 1) \leq 0$, with equality
only if $\#(\pi\gamma) = n/2 + 1$ and $\epsilon_r = -
\epsilon_s$ for all $(r, s) \in \pi$.

\end{lemma}

\begin{proof}
$2\#( \epsilon \gamma \delta \gamma^{-1} \epsilon \vee \pi
  \delta \pi \delta) = \#(\gamma \delta \gamma^{-1} \delta
  \epsilon \pi \delta \pi \epsilon)$. Now $\gamma \delta
  \gamma^{-1} \delta$ has 2 cycles and $\epsilon \pi \delta
  \pi \epsilon$ is a pairing. Thus $\#(\gamma \delta
  \gamma^{-1} \delta) = 2$ and $\#(\epsilon \pi \delta \pi
  \epsilon) = n$.

Now we consider two cases. In the first case for all $(r,
s) \in \pi$ we have $\epsilon_r = -\epsilon_s$. Then
$\epsilon \pi \delta \pi \epsilon = \pi \delta \pi
\delta$. In this case
\[
\#(\gamma \delta \gamma^{-1} \delta \epsilon \pi \delta \pi \epsilon)
=
\#(\gamma \delta \gamma^{-1} \delta \pi \delta \pi \delta)
=
\#(\gamma\pi \delta \gamma^{-1}\pi \delta)
 = 2 \#(\gamma \pi)
\]
Note that we have used the fact that $\delta \gamma^{-1}
\delta$ and $\pi$ act non-trivially only on disjoint sets
and thus commute. Thus for $\sigma \in S_n$ we have
$\#(\sigma \delta \sigma^{-1} \delta) = 2 \#(\sigma)$. These
two facts will be used a number of times below.  So by
Eq. (\ref{eq:euler_characteristic}) we have for some $g \geq
0$
\[
\#(\pi) + \#(\gamma\pi) + \#(\gamma) = n + 2(1 - g).
\]
So
\[
\#(\gamma\pi) = n/2 + 1 - 2g.
\]
Thus
\begin{eqnarray*}\lefteqn{
\#( \epsilon \gamma \delta \gamma^{-1} \epsilon \vee \pi
\delta \pi \delta) \ab -(n/2 + 1) }\\
&& = \#(\gamma\pi) - (n/2 +
1) = -2 g \leq 0.
\end{eqnarray*}

In the second case there is some $(r, s) \in \pi$ such that
$\epsilon_r = \epsilon_s$. In this case $\langle \gamma
\delta\gamma^{-1} \delta, \epsilon \pi \delta \pi \epsilon
\rangle$ acts transitively on $[\pm n]$. So again by
Eq. (\ref{eq:euler_characteristic}) we have for some $g' \geq 0$
\[
\#( \gamma \delta \gamma^{-1} \delta \epsilon \pi \delta \pi
\epsilon ) + \#(\epsilon \pi \delta \pi \epsilon ) +
\#(\gamma \delta \gamma^{-1} \delta) = 2n + 2(1 - g').
\]
Thus
\[
\#( \epsilon \gamma \delta \gamma^{-1} \epsilon \vee \pi
\delta \pi \delta) \ab -(n/2 + 1) = -1 -  g' \leq -1.
\]
\end{proof}

\begin{remark}\label{remark:highest_order}
We have identified the leading term as all the pairs $(\pi,
\epsilon)$ where $\#(\gamma\pi) = n/2+1$ and $\epsilon_r = -
\epsilon_s$ for all $(r, s) \in \pi$. The first condition is
that $\pi \in NC_2(n)$. Since there are for a given $\pi$,
$2^{n/2}$ ways of choosing $\epsilon$ so that the second
condition is satisfied we get, as expected, that the leading term of
$\E(\tr(X^n))$ is the Catalan number $|NC_2(n)| = C_{n/2} =
\frac{1}{n/2 + 1} \binom{n}{n/2}$, i.e.
\[
m_n = \lim_{n \rightarrow \infty} \E(\tr(X^n)) = C_{n/2}.
\]
Thus $N( \E(\tr(X^n)) - C_{n/2})$ starts with the
coefficient of $N^{-1}$ in the expansion
(\ref{equation:1}). Hence $m'_n = \lim_N N( \E(\tr(X^n)) -
C_{n/2})$ is the coefficient of $N^{-1}$. Suppose that $\pi
\in \cP_2(n)$ and $\epsilon_r = -\epsilon_s$ for all $(r, s)
\in \pi$. Then as noted above we have $\epsilon\pi \delta
\pi \epsilon = \pi \delta \pi \delta$ so for some $g \geq 0$
\[
\#(\epsilon \gamma \delta \gamma^{-1} \epsilon \vee \pi\delta\pi\delta) - (n/2 + 1) = -2g. 
\]
Thus these pairs cannot contribute to the coefficient of $N^{-1}$.

\begin{corollary}\label{cor:sub_leading_order}
The only pairs $(\pi, \epsilon)$ that can contribute to the
coefficient of $N^{-1}$ in (\ref{equation:1}) are those for
which there is at least one pair $(r, s) \in \pi$ such that
$\epsilon_r = \epsilon_s$ and $\#(\gamma \delta \gamma^{-1}
\delta\epsilon \pi \delta \pi \epsilon ) = n$.
\end{corollary}

\begin{proof}
We saw that to contribute to the $N^{-1}$ term we must have
at least one pair $(r, s) \in \pi$ such that $\epsilon_r =
\epsilon_s$ and
\[
\#( \epsilon \gamma \delta \gamma^{-1} \epsilon \vee \pi
\delta \pi \delta) \ab -(n/2 + 1) = -1.
\]
This implies $\#(\gamma \delta \gamma^{-1}
\delta\epsilon \pi \delta \pi \epsilon ) = n$.
\end{proof}

\begin{remark}
As we have observed in the calculations above, for a given
pair $(\pi, \epsilon)$ all that matters for the permutation
$\epsilon\pi \delta \pi \epsilon$, and thus the right hand
side of Eq. (\ref{equation:1}), is whether for each pair
$(r, s) \in \pi$ we have (\textit{a}) $\epsilon_r = -
\epsilon_s$ or (\textit{b}) $\epsilon_r = \epsilon_s$. For a
given $\pi$ and a choice of (\textit{a}) or (\textit{b}) for
each pair, there are $2^{n/2}$ choices of $\epsilon$.
\end{remark}

In the next section we shall show that $m'_n$ counts a
certain number of planar diagrams. The first five
non-zero infinitesimal moments are:
\begin{center}\begin{tabular}{c|ccccc}
$n$     & 2 & 4 &  6 &  8 & 10\\ \hline
$m'_n$  & 1 & 5 & 22 & 93 & 386  \\
\end{tabular}. \end{center}

\section{Infinitesimal moments and non-crossing partitions}
\label{section:inf_moments}
In this section we present the non-crossing partitions that
describe the infinitesimal moments of the \goe{}.  For
example $m'_4 = 5$ and the five diagrams are in Figure
\ref{fig:fourth_inf_moment}.

\setbox1=\hbox{\includegraphics{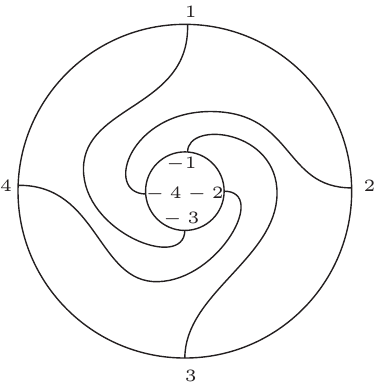}}
\setbox2=\hbox{\includegraphics{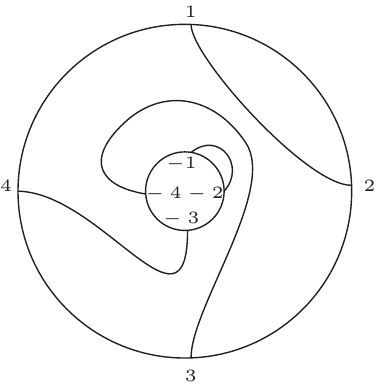}}
\setbox3=\hbox{\includegraphics{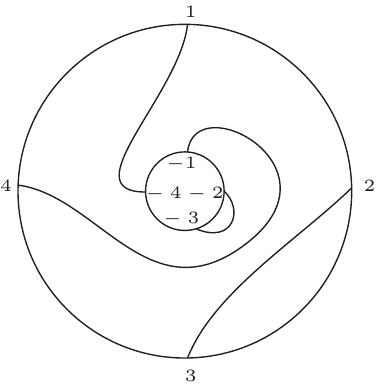}}
\setbox4=\hbox{\includegraphics{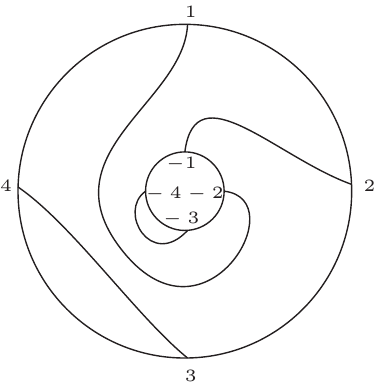}}
\setbox5=\hbox{\includegraphics{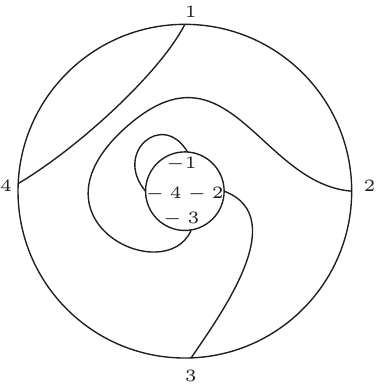}}

\begin{figure}
\leavevmode
\box1\hfill\box2\hfill\box3

\leavevmode
\hfill\box4\hfill\box5\hfill\hbox{}
\caption{\small\label{fig:fourth_inf_moment} The 5
  non-crossing annular pairings corresponding the
  infinitesimal moment $m'_4 = 5$. Note that if $(r,s)$ is a
  pair then so is $(-r, -s)$.}
\end{figure}

By Corollary \ref{cor:sub_leading_order} we must find all
pairs $(\pi, \epsilon)$ with $\pi \in \cP_2(n)$ and
$\epsilon \in \bZ_2^n$ such that there is at least one pair
$(r, s) \in \pi$ such that $\epsilon_r = \epsilon_s$ and
$\#(\gamma \delta \gamma^{-1} \delta \rho) = n$ where $\rho
= \epsilon \pi \delta \pi \epsilon$. These are exactly the
non-crossing annular pairings of \cite[Thm. 6.1]{mn} where
we have reversed the orientation of the inner circle (see
Figure \ref{fig:fourth_inf_moment}). Moreover we do not get
all non-crossing annular pairings, only those for which
\begin{enumerate}

\item
$\rho$ commutes with $\delta$,

\item
For all $r \in [n]$, $(r, -r)$ is never a pair of $\rho$.

\item
the blocks of $\rho$ come in pairs:  if $(r, s) \in \pi$ then
$(-r, -s) \in\rho$. 
\end{enumerate}
If $r$ and $s$ have opposite signs then $(r,s)$ is a
\textit{through string}, i.e. it connects the two
circles. Thus these pairings always connect the two circles.
\end{remark}

\begin{notation}
We denote by $NC_2^\delta(n,-n)$ the set of all non-crossing
annular pairings $\rho$ that satisfy (\textit{i}),
(\textit{ii}), (\textit{iii}) above. By convention
$NC_2^\delta(n,-n)$ is empty for $n$ odd.
\end{notation} 
We have put a `$-$' sign in front of the second `$n$' to
remind us the orientation of the inside circle has been
reversed from that used in \cite{mn}.  Summarizing the discussion above we have the
following theorem.

\begin{theorem}\label{thm:inf_moments_diagrams}
Let $X_N$ be the $N \times N$ \goe{} and $m_n$ the $n^{th}$
moment of the semi-circle law. Then the infinitesimal
moments of the \goe{} are given by $m'_n = \lim_N
N(\E(\tr(X_N^n)) - m_n) = |NC_2^\delta(n,-n)|$ for $n$ even
and $m' = 0$ for $n$ odd.
\end{theorem}

\setbox4=\hbox{%
\begin{tikzpicture}[anchor=base, baseline]
\draw [ultra thick] (0.0 , 0.5) -- (0.0, -0.2) -- (2.25 , -0.2) -- (2.25, 0.5);
\node[above] at (0.0,.45) {$+$};
\node[above] at (0.0,.75) {1};
\node[above] at (2.25,.45) {$+$};
\node[above] at (2.25,.75) {7};
\draw [thick] (0.375 , 0.5) -- (0.375, 0.2) -- (0.75 , 0.2) -- (0.75, 0.5);
\node[above] at (0.375,.45) {$+$};
\node[above] at (0.375,.75) {2};
\node[above] at (0.75,.45) {$-$};
\node[above] at (0.75,.75) {3};
\draw [ultra thick] (1.125 , 0.5) -- (1.125, 0.0) -- (2.625, 0.0) -- (2.625, 0.5);
\node[above] at (1.125,.45) {$+$};
\node[above] at (1.125,.75) {4};
\node[above] at (2.625,.45) {$+$};
\node[above] at (2.625,.75) {8};
\draw [thick] (1.5 , 0.5) -- (1.5, 0.2) -- (1.875, 0.2) -- (1.875, 0.5);
\node[above] at (1.5,.45) {$+$};
\node[above] at (1.5,.75) {5};
\node[above] at (1.875,.45) {$-$};
\node[above] at (1.875,.75) {6};
\end{tikzpicture}}


\setbox6=\hbox{%
\begin{tikzpicture}[anchor=base, baseline]
\draw [ultra thick] (0.0 , 0.5) -- (0.0, -0.2) -- (0.75 , -0.2) -- (0.75, -0.9);
\node[above] at (0.0,.45) {$+$};
\node[above] at (0.0,.75) {1};
\node[above] at (0.75,-1.4) {$+$};
\node[above] at (0.75,-1.8) {7};
\draw [thick] (0.375 , 0.5) -- (0.375, 0.2) -- (0.75 , 0.2) -- (0.75, 0.5);
\node[above] at (0.375,.45) {$+$};
\node[above] at (0.375,.75) {2};
\node[above] at (0.75,.45) {$-$};
\node[above] at (0.75,.75) {3};
\draw [ultra thick] (1.125 , 0.5) -- (1.125, -0.9);
\node[above] at (1.125,.45) {$+$};
\node[above] at (1.125,.75) {4};
\node[above] at (1.125,-1.4) {$+$};
\node[above] at (1.125,-1.8) {8};
\draw [thick] (0.0 , -0.9) -- (0.0, -0.6) -- (0.375, -0.6) -- (0.375, -0.9);
\node[above] at (0.0,-1.4) {$+$};
\node[above] at (0.0,-1.8) {5};
\node[above] at (0.375,-1.4) {$-$};
\node[above] at (0.375,-1.8) {6};
\end{tikzpicture}}

\setbox5=\hbox{\includegraphics[scale=0.7]{example_2}}

\begin{figure}[t]
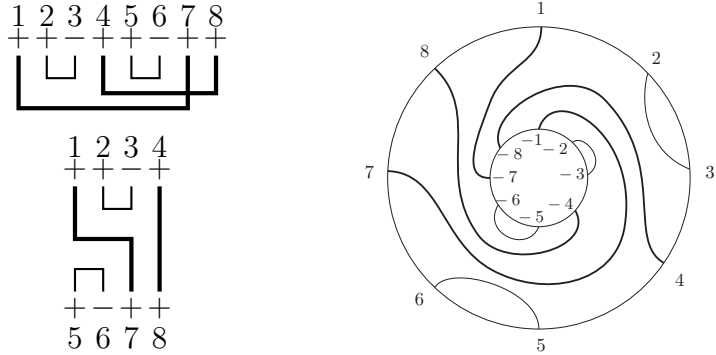

\begin{center}
\hfill$\vcenter{\hsize\wd4\box4 \medskip\hfill\box6\hfill\hbox{}}$ \hfill $\vcenter{\hsize\wd5\box5}$\hfill\hbox{}
\caption{\small\label{fig:1} We illustrate here an element
  of $NC_2^\delta(8, -8)$. Let $\pi =
  (1,7)\ab(2,3)\ab(4,8)(5,6)$, and $\epsilon = (1,
  1,$$-1,1,\ab 1, -1,\ab 1, 1)$. Then $\rho = \epsilon \pi
  \delta \pi \epsilon = (1, -7)\ab(-1, 7)(2, 3)(-2, -3) (4,
  -8)(-4, 8)\ab (5,6)\ab (-5,-6)$. Note that the symmetry
  condition $\delta \rho = \rho \delta$ means that once the
  non-through strings are placed (in this example $(2,
  4)(5,6)(-2, -4)(-5, -6)$) the through strings are forced;
  i.e. we must pair $4$ with $-8$ etc.  Note that by reversing the order of $\{5,6,7,8\}$ we can make the diagram non-crossing, see Remark \ref{rem:twisted_pairings}. {\vrule width 0pt
    depth 2em}}
\end{center}
\end{figure}

\setbox1=\hbox{%
\begin{tikzpicture}[anchor=base, baseline]
\draw [ultra thick] (0.0 , 0.5) -- (0.0, 0.0) -- (0.5 , 0.0) -- (0.5, 0.5);
\node[above] at (0.0,.75) {1};
\node[above] at (0.0,.45) {$+$};
\node[above] at (0.5,.75) {2};
\node[above] at (0.5,.45) {$+$};
\end{tikzpicture}}

\setbox2=\hbox{%
\begin{tikzpicture}[anchor=base, baseline]
\draw [ultra thick] (0.0 , 0.5) -- (0.0, 0.2) -- (0.75 , 0.2) -- (0.75, 0.5);
\node[above] at (0.0,.75) {1};
\node[above] at (0.0,.45) {$+$};
\node[above] at (0.75,.75) {3};
\node[above] at (0.75,.45) {$+$};
\draw [ultra thick] (0.375 , 0.5) -- (0.375, 0.0) -- (1.125 , 0.0) -- (1.125, 0.5);
\node[above] at (0.375,.75) {2};
\node[above] at (0.375,.45) {$+$};
\node[above] at (1.125,.75) {4};
\node[above] at (1.125,.45) {$+$};
\end{tikzpicture}}

\setbox3=\hbox{%
\begin{tikzpicture}[anchor=base, baseline]
\draw [ultra thick] (0.0 , 0.5) -- (0.0, 0.2) -- (1.125 , 0.2) -- (1.125, 0.5);
\node[above] at (0.0,.75) {1};
\node[above] at (0.0,.45) {$+$};
\node[above] at (1.125,.75) {4};
\node[above] at (1.125,.45) {$+$};
\draw [ultra thick] (0.375 , 0.5) -- (0.375, 0.0) -- (1.5 , 0.0) -- (1.5, 0.5);
\node[above] at (0.375,.75) {2};
\node[above] at (0.375,.45) {$+$};
\node[above] at (1.5,.75) {5};
\node[above] at (1.5,.45) {$+$};
\draw [ultra thick] (0.75 , 0.5) -- (0.75, -0.2) -- (1.875, -0.2) -- (1.875, 0.5);
\node[above] at (0.75,.75) {3};
\node[above] at (0.75,.45) {$+$};
\node[above] at (1.875,.75) {6};
\node[above] at (1.875,.45) {$+$};
\end{tikzpicture}}

Here are three examples of $(\pi, \epsilon)$'s with all
through strings.

\leavevmode\hfill
$\vcenter{\hsize=\wd1\box1}$
or
$\vcenter{\hsize=\wd2\box2}$
or
\raise -1.4em\box3
\hfill\hbox{}

\noindent
Here is an example with 6 through strings and 16 non-through
strings.

\leavevmode\hfill
\begin{tikzpicture}[anchor=base, baseline]
\draw [thick](0.375, 0.5) -- (0.375, 0.2) -- (0.75, 0.2) -- (0.75, 0.5);
\node[above] at (0.375,.45) {$+$};
\node[above] at (0.75,.45) {$-$};
\draw [ultra thick] (1.125 , 0.5) -- (1.125, -0.2) -- (4.5 , -0.2) -- (4.5, 0.5);
\node[above] at (1.125,.45) {$+$};
\node[above] at (4.5,.45) {$+$};
\draw [thick](1.5, 0.5) -- (1.5, 0.0) -- (2.625, 0.0) -- (2.625, 0.5);
\node[above] at (1.5,.45) {$+$};
\node[above] at (2.625,.45) {$-$};
\draw [thick](1.875, 0.5) -- (1.875, 0.2) -- (2.25, 0.2) -- (2.25, 0.5);
\node[above] at (1.875,.45) {$+$};
\node[above] at (2.25,.45) {$-$};
\draw [ultra thick] (3 , 0.5) -- (3, -0.4) -- (6.375 , -0.4) -- (6.375, 0.5);
\node[above] at (3,.45) {$+$};
\node[above] at (6.375,.45) {$+$};
\draw [thick] (3.375, 0.5) -- (3.375, 0.2) -- (3.75, 0.2) -- (3.75, 0.5);
\node[above] at (3.37,.45) {$+$};
\node[above] at (3.75,.45) {$-$};
\draw [ultra thick] (4.125 , 0.5) -- (4.125, -0.6) -- (6.75 , -0.6) -- (6.75, 0.5);
\node[above] at (4.125,.45) {$+$};
\node[above] at (6.75,.45) {$+$};
\draw [thick] (4.875, 0.5) -- (4.875, 0.2) -- (5.25, 0.2) -- (5.25, 0.5);
\node[above] at (4.875,.45) {$+$};
\node[above] at (5.25,.45) {$-$};
\draw [thick](5.625, 0.5) -- (5.625, 0.2) -- (6, 0.2) -- (6, 0.5);
\node[above] at (5.625,.45) {$+$};
\node[above] at (6,.45) {$-$};
\draw [thick](7.125, 0.5) -- (7.125, 0.0) -- (8.375, 0.0) -- (8.375, 0.5);
\node[above] at (7.125,.45) {$+$};
\node[above] at (8.375,.45) {$-$};
\draw [thick](7.5, 0.5) -- (7.5, 0.2) -- (8, 0.2) -- (8, 0.5);
\node[above] at (7.5,.45) {$+$};
\node[above] at (8.0,.45) {$-$};
\end{tikzpicture}
\hfill\hbox{}

\medskip
\noindent
$\pi = \{(1,2)(3, 12)(4,7)(5,6)(8, 17)(9, 10)(11, 18)(13, 14)(15, 16)\ab(19, 22)\ab(20, 21)\}$.

\begin{remark}\label{rem:alternative_description}
Scrutiny of these examples reveals an important
alternative way of describing elements of $NC_2^\delta(n,
-n)$ that will be useful in computing the infinitesimal
cumulants $\{\kappa'_n\}_n$ of $(\mu, \mu')$. The important
property is that if we fuse the thick lines, formed by the
through strings, we always get a non-crossing
partition. Moreover the thick lines always occur in the same
way: if the block formed by the thick lines is $(i_1, \dots,
i_k)$ then $k =2 l$ must be even and the pairs are $(i_1,
i_{l+1})$, $(i_2, i_{l+2})$, \dots, $(i_l, i_{2l})$. Thus
given a non-crossing partition $\pi \in NC(n)$ and a block
$V \in \pi$ such that $|V|$ is even and all other blocks of
$\pi$ have 2 elements we can construct an element of
$NC_2^\delta(n, -n)$.

\end{remark}

\begin{definition}[{c.f.\cite[Def.\,1]{amsw} and \cite[\S6]{kms}}]
Let $\pi$ be a non-crossing partition in which no block has
more than two elements. From $\pi$ create a new partition
$\tilde\pi$ by joining into a single block all the blocks of
$\pi$ of size 1, call this block $V$. If $\tilde\pi$ is
non-crossing and $|V|$ is even we say that $(\tilde\pi, V)$
is a non-crossing \textit{half-pairing}. The blocks of $\pi$
of size 1 are called the \textit{through strings}. (See
Figure \ref{fig:2}.) We let $NCC_2(n) = \{ (\pi, V) \mid \pi
\in NC(n), V \in \pi$, $|V|$ is even, and all other blocks
of $\pi$ have 2 elements$\}$. By convention $NCC_2(n)$ is
empty for $n$ odd.

\end{definition}

\begin{figure}
\hbox{}\hfill
\includegraphics{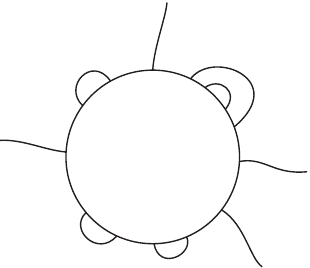} \hfill
\includegraphics{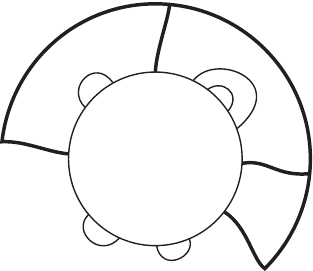}
\hfill\hbox{}
\caption{
{\small  \label{fig:2}
On the left is $\pi$ and on the right is $\tilde \pi$. }} 
\end{figure}

\begin{remark}\label{remark:bijection}
From Remark \ref{rem:alternative_description} we see that
there is a bijection from $NCC_2(n)$ to $NC_2^\delta(n, -n)$
where a pair $(\pi, V)$ with $V$ a block of size $k$
produces a $\rho \in NC_2^\delta(n, -n)$ with $k$ through
strings. By \cite[Lemma 13]{amsw} the number of non-crossing
half-pairings with $k$ through strings is
$\binom{n}{(n-k)/2}$
\end{remark}

\begin{remark}\label{rem:twisted_pairings}
In Figure \ref{fig:1} we presented an example where we have $(\pi, \epsilon)$, a pairing of $[8]$ with crossings and an assignment of signs, we unfolded the diagram into $\rho$,  a non-crossing annular pairing and also a non-crossing pairing on a disc. This is in fact a general situation. We rotate the numbers $\{ 1, 2, 3, \dots, n\}$ until half the thick lines are in $\{ 1, 2, 3, \dots, n/2\}$. Then reverse the numbers $\{n/2 + 1, \dots, n\}$.
\end{remark}

\begin{lemma}\label{lemma:inf_moments}
Let $n = 2m$. The number of non-crossing annular pairings
satisfying (\textit{i}), (\textit{ii}), and (\textit{iii})
above is
\[
m'_n = \sum_{k=1}^{m} \binom{n}{m-k} = \frac{1}{2} \Big(
2^n - \binom{n}{m}\Big).
\]
\end{lemma}

\begin{proof}
We know that the number of non-crossing annular pairings of
an $(p,q)$-annulus with $l$ through strings is $l
\binom{p}{\frac{p-l}{2}} \binom{q}{\frac{q-l}{2}}$ (see
e.g. \cite[Eq. (11)]{mst}). In our situation $p = q = 2m$,
$l = 2k$ is even and the pairings on the circles are
symmetric so we only get $\binom{n}{m-k}$ diagrams, because
once the non-through strings are placed there is only one way
to place the through strings.
\begin{eqnarray*}\lefteqn{
2 \sum_{k=1}^m \binom{n}{m-k} 
= 
2\bigg\{ \binom{n}{0} + \binom{n}{1} + \cdots + \binom{n}{m-1}\bigg\}} \\
&=&
\binom{n}{0}  + \cdots + \binom{n}{m-1}
+
\binom{n}{m+1} + \cdots + \binom{n}{n-1} +  \binom{n}{n} \\
&=&
2^n - \binom{n}{m}.
\end{eqnarray*}
\end{proof}

\begin{theorem}
Let $\nu_1 = \frac{1}{2}(\delta_{-2} + \delta_2)$ be the
Bernoulli distribution and $d\nu_2(t) = \frac{1}{\pi}
\frac{1}{\sqrt{4 - t^2}} \, dt$ be the arcsine law. Let
$\mu' = \frac{1}{2}(\nu_1 - \nu_2)$. Let $X_N$ be the $N
\times N$ \goe{} and $m_n$ the $n^{th}$ moment of the
semi-circle law. Then
\[
\lim_N N(\E(\tr(X^n)) - m_n) = m_n'
\]
and
\[
m_n' = \int t^n \, d\mu'(t).
\]
\end{theorem}

\section{Infinitesimal cumulants of the goe}
\label{section:inf_cum_goe}
Recall that if a partition has all blocks of even size then
it is called an \textit{even partition}. We already know
that for the infinitesimal \goe{} we have $\kappa_2 = 1$ and
all other $\kappa_n = 0$ (i.e. the semi-circle law). We
shall show that $\kappa'_{2m} =1$ and $\kappa'_{2m-1} = 0$
for all $m \geq 1$. This means that for an even non-crossing
partition of $[n]$ (using the notation of
Eq. (\ref{eq:diff_cumulant}))
\begin{equation}\label{eq:inf_cum_part_goe}
\partial\kappa_\pi =
\begin{cases}
\frac{n}{2} & \pi \mbox{\ is a pairing}\\
1 & \pi\ \vtop{\hsize
  180pt\noindent\raggedright\baselineskip=12pt has one block
  of size $k$ and all others of size 2} \\
0 & \pi\ \vtop{\hsize
  180pt\noindent\raggedright\baselineskip=12pt has at least
  2 blocks with more than 2 elements}
\end{cases}
\end{equation}
From Kreweras \cite[Th\'eor\`eme 4]{krew} we know that that
number of partitions with one block of size $k > 2$ and all
others of size 2 is $\binom{n}{\frac{n+k}{2}}$.

Using the cumulant-moment formula
(Eq.~(\ref{eq:infinitesimal_cumulant_moment})) we have
\[
\kappa'_1 = m'_1
\]
\[
\kappa'_2 = m'_2 - 2 m'_1 m_1
\]
\[
\kappa'_3 = m'_3 - 3m'_1 m_2 - 3 m_1 m'_2 + 6 m'_1 m_1 m_1
\]
\begin{multline*}
\kappa'_4 = m'_4 - 4m'_1 m_3 - 4 m_1 m'_3 - 4 m'_2 m_2 + 12 m'_2 (m_1)^2 \\
 + 24 m'_1 m_2 m_1 - 24 m'_1 (m_1)^3
\end{multline*}
All these formulas can be obtained by assuming an implicit
dependence on a parameter $t$ and applying $\ds \frac{d}{dt}\Big|_{t=0}$ to both sides of
\[
\kappa_n = \sum_{\pi \in NC(n)} \kappa_\pi
\mbox{\ with\ }
\kappa'_n = \ds \frac{d}{dt}\Big|_{t=0} \kappa_n 
\mbox{\ and\ }
m'_n = \ds \frac{d}{dt}\Big|_{t=0} m_n. 
\]
Since $m'_1 = m'_3 = 0$ and $m'_2 = 1$ and $m'_4 = 5$ we
have $\kappa'_1 = \kappa'_3 = 0$ and $\kappa'_2 = \kappa'_4
= 1$. 
\begin{lemma}\label{lemma:odd_inf_cuml}
If $n$ is odd then $\kappa'_n = 0$.
\end{lemma}
\begin{proof}
Let us recall some standard notation. For $\pi \in \cP(n)$
we let $m_\pi = \prod_{V \in \pi} m_{|V|}$. We set
\[
\partial m_\pi = \sum_{V \in \pi} m'_{|V|} 
\mathop{\prod_{W \in \pi}}_{W \not = V}
m_{|W|}
\]
The moment-cumulant formula
\[
\kappa_n = \sum_{\pi \in NC(n)} \mu(\pi, 1_n) m_\pi
\]
becomes
\[
\kappa'_n = \sum_{\pi \in NC(n)} \mu(\pi, 1_n) \partial m_\pi.
\]
We have seen that both $m_n = 0$ and $m'_n = 0$ for $n$
odd. If $n$ is odd and $\pi \in \cP(n)$ then $\pi$ must have
a block of odd size. Thus $\partial m_\pi = 0$. Hence for
$n$ odd $\kappa'_n = 0$.
\end{proof}

\begin{theorem}\label{thm:inf_goe_cum}
For $n$ even $\kappa'_n = 1$.
\end{theorem}

\begin{proof}
We have already shown that $\kappa'_2 = \kappa'_4 =
1$. Suppose that we have shown that $\kappa'_2 = \cdots =
\kappa'_{n-2} = 1$. We shall prove that $\kappa'_n =
1$. From the infinitesimal moment-cumulant formula
(\ref{eq:inf_moment_cumulant}) we have that
\[
m'_n = \sum_{\pi \in NC(n)} \sum_{V \in \pi}
\kappa'_{\pi, V}. 
\]
By induction we have that $\kappa'_{\pi, V} = 1$ for $(\pi,
V) \in NCC_2(n)$ and $\pi \not= 1_n$. Moreover by Lemma
\ref{lemma:odd_inf_cuml} we have that $\kappa'_{\pi, V} = 0$
if $(\pi, V) \not\in NCC_2(n)$. Hence
\[
m'_n - \kappa'_n = 
\mathop{\sum_{\pi \in NC(n)}}_{\pi \not= 1_n}
 \sum_{V \in \pi} \kappa'_{\pi, V} = |NCC_2(n)| - 1.
\] 
Since we have by Lemma \ref{lemma:inf_moments} that $m'_n =
|NCC_2(n)|$ we must have $\kappa'_n = 1$ as claimed.
\end{proof}

\begin{corollary}\label{cor:inf_r_transform}
For the infinitesimal \goe{} we have the infinitesimal
$r$-transform is given by $r(z) = \frac{z}{1 - z^2}$.
\end{corollary}

\begin{remark}
By Corollary \ref{cor:inf_r_transform} and
Eq. (\ref{eq:infinitesimal_r-transform}) we have that the
infinitesimal Cauchy transform of $\mu'$ is
\[
g(z) = - r(G(z)) G'(z) = \frac{G(z)}{z^2 - 4} = \frac{z -
  \sqrt{z^2 - 4}}{2(z^2 - 4)}.
\]
Note that in accordance with
Eq.~(\ref{eq:infinitesimal_goe}), $g$ has poles at $z = 2$
and $z = -2$ each with residue $\frac{1}{4}$.
\end{remark}

Now that we have the infinitesimal cumulants we can easily
see that independent \goe{}'s cannot be asymptotically free. 

\begin{proposition}\label{cor:nonfreeness_goe}
Let $\{X_N\}_N$ and $\{Y_N\}_N$ be independent ensembles of
\goe{} random matrices. Then $\{X_N\}_N$ and $\{Y_N\}_N$ are
not asymptotically infinitesimally free.
\end{proposition}

\begin{proof}
Suppose $\{X_N\}_N$ and $\{Y_N\}_N$ were asymptotically
infinitesimally free. Then there would be an infinitesimal
probability space $(\cA, \phi, \phi')$ and $x, y \in \cA$
which are infinitesimally free such that for all $k$
\[
\lim_N \E(\tr(X_N^k)) = \phi(x^k) \mathrm{\ and\ } 
\lim_N N( \E(\tr(X_N^k)) - \phi(x^k)) = \phi'(x^k)
\]
\hfill and\hfill\hbox{}
\[
\lim_N \E(\tr(Y_N^k)) = \phi(y^k) \mathrm{\ and\ } 
\lim_N N( \E(\tr(Y_N^k)) - \phi(y^k)) = \phi'(y^k).
\]
Let $Z_N = \frac{1}{\sqrt{2}}( X_N + Y_N)$ and $z =
\frac{1}{\sqrt{2}}(x + y)$. Then by Remark
\ref{rem:independent_sum}, $\{Z_N\}_N$ is also a \goe{}
random matrix and so by Theorem \ref{thm:inf_goe_cum}, for
all $n$
\[
{\kappa'}_{2n}^{(x)} = {\kappa'}_{2n}^{(y)} = {\kappa'}_{2n}^{(z)} = 1.
\]
On the other hand by our assumption of infinitesimal
freeness we have by the vanishing of mixed cumulants
(Thm. \ref{thm:vanishing_mixed_cumulants}) that
${\kappa'}_{2n}^{(z)} = 2^{-n +1}$. This contradiction shows
that the ensembles $\{X_N\}_N$ and $\{Y_N\}_N$ cannot be
asymptotically infinitesimally free. \end{proof}

\begin{remark}
Independent \goe{} random matrices are not asymptotically
second order free but are asymptotically real second order
free (see Redelmeier \cite{r}). Thus there may be a positive
statement one can make in the orthogonal case, see Remark
\ref{remark:possible_result}.
\end{remark}

\section{Asymptotic infinitesimal freeness\\
         for complex wishart matrices}\label{section:asymptotic}

To discuss asymptotic infinitesimal freeness we shall make
use of the algebra $\bC \la Y_1, \dots, Y_s \ra$ of
polynomials in the non-commuting variables $Y_1, \dots,
Y_s$. Given elements $x_1, \dots, x_s$ in an infinitesimal
probability space $(\cA, \phi, \phi')$ we get two linear
functionals on $\bC\la Y_1, \dots, Y_s\ra$ given by
$\mu(Y_{i_1} \cdots Y_{i_n}) = \phi(x_{i_1} \cdots x_{i_n})$
and $\mu'(Y_{i_1} \cdots T_{i_n}) = \phi'(x_{i_1} \cdots
x_{i_n})$ for all $i_1, \dots, i_n \in [s]$. We call the
pair $(\mu, \mu')$ the \textit{algebraic infinitesimal
  distribution} of $x_1, \dots, x_s$.

For example if $X^{(N)}_1, \cdots, X^{(N)}_s$ are
independent $N \times N$ complex Wishart random matrices
then using \cite[Cor. 9.6]{mn}, we can use the $1/N$
expansion of $\E(\tr(X^{(N)}_{i_1}\ab \cdots
X^{(N)}_{i_n}))$ to define a pair $(\mu_N, \mu'_N)$.  We let
for $i_1, \dots, i_n \in [s]$
\begin{equation}\label{eq:mp_limit_dist}
\mu_N(Y_{i_1} \cdots Y_{i_n}) =
\E(\tr(X_{i_1} \cdots X_{i_n})) 
\mbox{\ and\ }
\mu(Y_{i_1} \cdots Y_{i_n}) =
\mathop{\sum_{\pi \in NC(n)}}_{\pi \leq \ker(i)}
c^{\#(\pi)}.
\end{equation}
Then we set
\begin{equation*}\mu'_N(I) = 0 \mathrm{\ and}
\end{equation*}
\begin{equation}\label{eq:inf_mom_cx_wish}
\mu'_N(Y_{i_1} \cdots Y_{i_n}) = 
N(\mu_N(Y_{i_1} \cdots Y_{i_n}) - \mu(Y_{i_1} \cdots Y_{i_n}))
\end{equation}
Finally we set
\begin{equation}\label{eq:inf_mixed_mom_cx_wish}
\mu'(Y_{i_1} \cdots Y_{i_n}) 
=
\mathop{\sum_{\pi \in NC(n)}}_{\pi \leq \ker(i)}%
c'\, \#(\pi)\, c^{\#(\pi)-1}.
\end{equation}

Recall here that $S_n$ is the symmetric group on $[n]$,
$\gamma_n = (1, 2, 3, \dots,\ab n)$, and $\#(\pi)$ is the
number of cycles in the cycle decomposition of $\pi$.  
Also $\ker(i) \in \cP(n)$ is the kernel of $i$ as described 
in Remark \ref{rem:independent_sum} page~\pageref{rem:independent_sum}. 

\begin{theorem}\label{thm:asy_inf_free_cx_wish}
Let $\{X^{(N)}_1, \dots, X^{(N)}_s\}_N$ be an independent
family of complex $N \times N$ Wishart matrices. Assume that
$\lim_N M/N = c$ and $\lim_N (M - Nc) = c'$. Then
$\{X^{(N)}_1, \dots, X^{(N)}_s\}$ are asymptotically
infinitesimally free and the infinitesimal cumulants of the
limit infinitesimal distribution are given by $\kappa'_n =
c'$ for all $n \geq 1$. The limit infinitesimal distribution
is $(\mu, \mu')$ where $\mu$ is the Marchenko-Pastur
distribution with parameter $c$ and $\mu'$ is given by
\begin{equation}\tag{2}
d\mu'(x) = -c'
\begin{cases}
\phantom{\frac{1}{2}}\delta_0 -  \frac{x + 1 -c}
{2 \strut\pi x \sqrt{(b - x)(x - a)}}\,dx & c < 1 \\
\frac{1}{2} \delta_0 -  \frac{1}{2\pi\sqrt{x(4 - x)}}\,dx & c = 1 \\
\phantom{\frac{1}{2} \delta_0}
-\frac{x + 1 -c}
{2 \strut\pi x \sqrt{(b - x)(x - a)}}\,dx  & c > 1 \\
\end{cases}
\end{equation}
where $a = (1 -  \sqrt c)^2$ and $b = (1 + \sqrt c)^2$. 
\end{theorem}

\begin{proof}

We shall show that
\[
\lim_N \mu_N = \mu \mbox{\ and\ } \lim_N \mu'_N = \mu'. 
\]
We have from \cite[Lemma 7.6]{mn}
\begin{multline}
\mu_N(Y_{i_1} \cdots Y_{i_n}) 
=
\mathop{\sum_{\pi \in S_n}}_{\pi \leq \ker(i)}
M^{\#(\pi)} N^{\#(\pi^{-1}\gamma) -(n + 1)}  \\
= 
\mathop{\sum_{\pi \in S_n}}_{\pi \leq \ker(i)}
\Big( \frac{M}{N}\Big)^{\#(\pi)}
N^{\#(\pi) + \#(\pi^{-1}\gamma) - (n + 1)}
\end{multline}
For
$\pi \in S_n$ we have $\#(\pi) + \#(\pi^{-1}\gamma_n) = n +
1 - 2g$ for some integer $g \geq 0$. 
Moreover the permutations $\pi$ for which $\#(\pi) +
\#(\pi^{-1}\gamma_n) = n + 1$ are exactly the non-crossing
partitions.  So for $\pi \in S_n \setminus NC(n)$ we have
$\#(\pi) + \#(\pi^{-1}\gamma) -n \leq -1$.  Thus
\[
\lim_{N \rightarrow \infty}
\mu_N(Y_{i_1} \cdots Y_{i_n}) =
\mu(Y_{i_1} \cdots Y_{i_n}).
\]
Since we have that $\pi \leq \ker(i)$ we have that mixed
cumulants vanish by the definition of $\mu$ in Equation
(\ref{eq:mp_limit_dist}). Thus we have that $X_1, \dots, X_s$ are
asymptotically free. Of course this is a known fact, see
\cite{cc}. Moreover we have
\begin{multline}
\mu_N'(Y_{i_1} \cdots Y_{i_n})
=
\mathop{\sum_{\pi \in NC(n)}}_{\pi \leq \ker(i)}
N \Big(\Big( \frac{M}{N}\Big)^{\#(\pi)} - c^{\#(\pi)}\Big) \\
\mbox{} + 
\mathop{\sum_{\pi \in S_n \setminus NC(n)}}_{\pi \leq \ker(i)}
\Big( \frac{M}{N}\Big)^{\#(\pi)} N^{\#(\pi) + \#(\pi^{-1}\gamma) - n}
\end{multline}
For any $\pi$ we have
\[
\lim_{N \rightarrow \infty} N \Big(\Big( \frac{M}{N}\Big)^{\#(\pi)} - c^{\#(\pi)}\Big)
= \#(\pi) c^{\#(\pi) -1} c'
\] 
and for $\pi \in S_n \setminus NC(n)$
\[
\lim_{N \rightarrow \infty}
\Big( \frac{M}{N}\Big)^{\#(\pi)} N^{\#(\pi) + \#(\pi^{-1}\gamma) - n}
= 0.
\]
Hence 
\[
\lim_{N \rightarrow \infty} \mu'_N(Y_{i_1} \cdots Y_{i_n})
= \mathop{\sum_{\pi \in NC(n)}}_{\pi \leq \ker(i)}
\#(\pi) c^{\#(\pi) -1 } c'
=
\mu'(Y_{i_1} \cdots Y_{i_n}).
\]
In the expression 
\[
\mu'(Y_{i_1} \cdots Y_{i_n}) =
\mathop{\sum_{\pi \in NC(n)}}_{\pi \leq \ker(i)}
\#(\pi) c^{\#(\pi) -1 } c'
\] 
the condition $\pi \leq \ker(i)$
shows that mixed infinitesimal cumulants also vanish. 
Thus $X_1, \dots, X_s$ are asymptotically infinitesimally 
free.

Note that when $\ker(i) = 1_n$ the right hand side of
Eq.~(\ref{eq:inf_mixed_mom_cx_wish}) is given by the
integral $\int t^n \,d\mu'(t)$ with this $\mu'$ the signed
measure in Eq.~(\ref{eq:inf_sing_mom_cx_wish}). Recall that
$\mu'(Y_1^n) = \sum_{\pi \in NC(n)} \partial \kappa_\pi$, so
Eq.~(\ref{eq:inf_mixed_mom_cx_wish}) shows that when
$\ker(i) = 1_n$ we have $\kappa'_n = c'$ for all $n$. 

Only Equation~(\ref{eq:inf_sing_mom_cx_wish}) remains to be
proved. We compute the infinitesimal Cauchy transform and
then use Stieltjes inversion. We have already shown that
$\kappa'_n = c'$ for all $n \geq 1$, thus $r(z) = c'/(1 -
z)$. Hence
\begin{multline*}
g(z) = - r(G(z)) G'(z)\\ = -c' \frac{G'(z)}{1 - G(z)} 
= \frac{-c'}{z P(z)}
\frac{(1 - c)^2 - (1 + c)z - (1 - c)P(z)}{P(z) + z -1 + c}  
\end{multline*}
where $P(z) = \sqrt{(z-a)(z-b)}$, and we choose the branch
as in \cite[Ex. 3.6]{ms2017}. Note that both $\big\{\mathop{\sum}_{\pi \in
  NC(n)} \big(\frac{M}{N}\big)^{\#(\pi)}\big\}_n$ and $\big\{
\sum_{\pi \in NC(n)} c^{\#(\pi)}\big\}_n$ are moment sequences
of positive measures, thus $g$ is the limit of the
difference of Cauchy transforms of positive measures and so
we can recover a signed measure by Stieltjes inversion. Now
let $Q(x) = \sqrt{(b - x)(x - a)}$ for $x \in [a, b]$ and
$Q(x) = 0$ for $x \not\in [a, b]$. For a function $f$ on
$\bC^+$ we let $f(x + i0^+) = \lim_{\epsilon \rightarrow
  0^+} f(x + i \epsilon)$. Then for $a \leq x \leq b$
\[
-\frac{1}{\pi} \im(g(x+i0^+)) = c' 
\frac{x + 1 - c}{2 \pi x Q(x)}. 
\]
As written above, $g$ has a singularity at $0$. For
$a \in \bR$ and $z \in \bC^+$ let us write $\lim_{\sa z \rightarrow a}$ for the non-tangential limit as $z$ approaches $a$ (see e.g \cite[p. 60]{ms2017}).
We have $\lim_{\sa z \rightarrow 0} P(z) + z - 1 + c = 0$ when $c > 1$,
so $ \lim_{\sa z \rightarrow 0} \frac{P(z) + z - 1 + c}{z} =
1 + P'(0) = 1 + \frac{1+c}{c -1} = \frac{2c}{c -1}$.  From
the equation for $g$ above we have
\[
z g(z) = \frac{c'(1 - c)}{P(z)}\Big\{ 1 + \frac{2c}{1 - c}
\frac{z}{P(z) + z - 1+ c}\Big\}
\]
and thus $\lim_{\sa z \rightarrow 0} z g(z) = 0$; so the
singularity at $0$ is removable. When $c < 1$ we have
$\lim_{\sa z \rightarrow 0} P(z) + z - 1 + c = 2(c-1)
\not=0$. Thus
\[
\lim_{\sa z \rightarrow 0}
\frac{z}{P(z) + z - 1+ c} = 0
\]
and hence $\lim_{\sa z \rightarrow 0} zg(z) =
\frac{c'(1-c)}{P(0)} =-c'$. When $c = 1$ we have
\[
zg(z) = \frac{2 c'}{z - 4 + P(z)}.
\]
and thus $\lim_{\sa z \rightarrow 0} zg(z) =
-\frac{c'}{2}$. Summarizing we have
\[
\lim_{\sa z \rightarrow 0} zg(z) =
\begin{cases}
-c'.          & c < 1 \\
-\frac{c'}{2} & c = 1 \\
0             & c > 1
\end{cases}
\]
We capture the weight of the mass at $0$ by
\cite[Prop.~3.8]{ms2017}. The singularities at $a$ and $b$
are removable.
\end{proof}

\begin{figure}\label{fig:densities}
\includegraphics[scale=0.7]{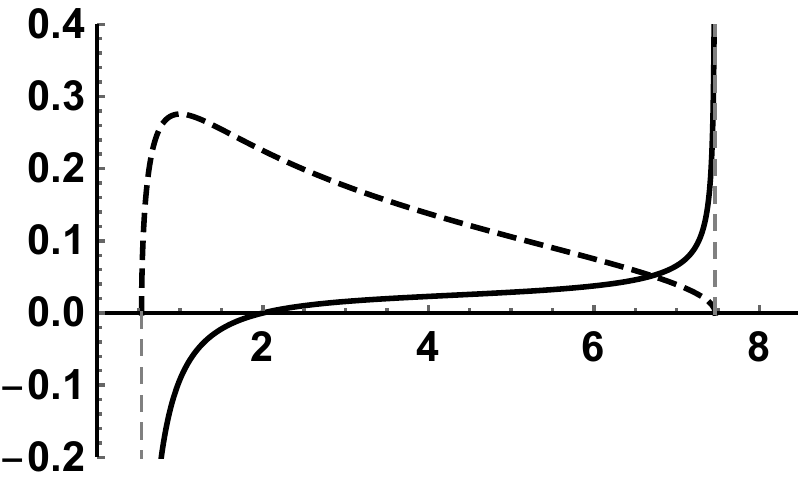}
\caption{The densities for $\mu$ (dashed) and $\mu'$ (solid) when $c = 3$ and $c' = 1$. }
\end{figure}

\begin{remark}\label{rem:formal_derivative}
To make the two distributions explicit let us summarize. For $c > 0$, $a = (1 - \sqrt{c})^2$ and $b = (1 + \sqrt{c})^2$ we have
\begin{equation}\label{eq:mp-density}
\int_a^b x^n \frac{\sqrt{(b - x)(x - a)}}{2 \pi x} \, dx
=
\sum_{\pi \in NC(n)} c^{\#(\pi)}
\end{equation}
\begin{equation}\label{eq:mp-inf-density}
\int_a^b x^n \frac{c '( x + 1 -c)}{ 2 \pi x \sqrt{(b - x)(x -a)}}
\, dx
= \sum_{\pi \in NC(n)} \#(\pi) c^{\#(\pi) -1 } c'.
\end{equation}
Note that we can also obtain (\ref{eq:mp-inf-density}) from
(\ref{eq:mp-density}) by formal differentiation by
$t$. Namely suppose that $c$ is an implicit function of $t$
and $c' = \frac{d}{dt}\big|_{t=0} c$. Then $(b - x)(x - a) =
-x^2 + 2(1 + c)x - (1 - c)^2$. So
\[
\frac{d}{dt}\Big|_{t=0} (b - x)(x - a) =
2 c'( x + 1 - c)
\]
and thus
\[
\frac{d}{dt}\Big|_{t=0}  \frac{\sqrt{(b - x)(x - a)}}{2 \pi x}
=
\frac{c '( x + 1 -c)}{ 2 \pi x \sqrt{(b - x)(x -a)}}.
\]
For $c \not= 1$, this formal operation picks up the mass at
$0$ if we say that $\frac{d}{dt}\big|_{t=0} (1 - c) \delta_0
= -c'\delta_0$. At $c = 1$ it seems a more delicate formal argument is
required. 
\end{remark}

\section{A universal rule for the \goe{} and constant matrices}
\label{section:universal}

We have already shown that independent \goe{} ensembles are
not asymptotically free. In this section we shall go a
little further and give a rule that shows a different kind
of freeness applies in the orthogonal case. First let us
recall a formula from \cite[Eq.(5.1)]{fn} for when $a_1,
\dots, a_n \in \cA_1$, $b_1, \dots, b_n \in \cA_2$ and
$\cA_1$ and $\cA_2$ are infinitesimally free
\begin{multline}\label{eq:inf_free_eqn}
\phi'(a_1b_1a_2b_2 \cdots a_nb_n) \\ 
= 
\sum_{\pi \in NC(n)} \{
\kappa_\pi(a_1, \cdots, a_n) \
                  \partial\phi_{K(\pi)}(b_1, \cdots, b_n) \\
\mbox{} +                          
\partial\kappa_\pi(a_1, \cdots, a_n)\ 
                          \phi_{K(\pi)}(b_1, \cdots, b_n)
                  \}
\end{multline} 
where $K(\pi)$ is the Kreweras complement of $\pi$. 

Suppose we have for each $N$, $A^{(N)}_1, \dots,
A^{(N)}_s$, $N \times N$ matrices  that have a joint limit 
$t$-distribution. Recall from \cite{mp} this means that
$\{A^{(N)}_1, A^{(N)\str}_1, \dots, A^{(N)}_s, A^{(N)\str}_s
\}$ has a joint limit distribution. Using our convention that
$A_i^{(1)} = A_i$ and $A_i^{(-1)} = A_i^\str$ this means
that for every $i_1, \dots, i_n$ and every $\epsilon_1,
\dots, \epsilon_s \in \{-1, 1\}$ the limit
\[
\lim_N \tr(A^{(N)(\epsilon_1)}_{i_1} \cdots A^{(N)(\epsilon_n)}_{i_n})
\]
exists; we denote this limit by $\phi(a_1^{(\epsilon_1)}
\cdots a_n^{(\epsilon_n)})$ where $a_1 \dots , a_s,
a_1^\str, \dots,\ab a^\str_s$ are in some non-commutative
infinitesimal space $(\cA, \phi, \phi')$ with a transpose
$a \mapsto a^\str$. Let us further suppose that $A_1, \dots, A_s$ have
a joint infinitesimal distribution. This means that for all
$i_1, \dots, i_n$ we have that
\begin{equation}\label{eq:infintesimal_limit}
\phi'(a_{i_1} \cdots a_{i_n}) = 
\lim_N N(\tr(A_{i_1} \cdots A_{i_n}) - \phi(a_{i_1} \cdots a_{i_n}))
\end{equation} 
exists. In order to describe the limiting behaviours we need
some notation.

\begin{figure}[t]\label{fig:noncrossing_annular_permutation}
\setbox1=\hbox{\includegraphics{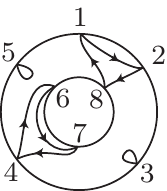}}
\hfill$\vcenter{\hsize\wd1\box1}$ \hfill $\vcenter{\hsize
  200pt\raggedright\noindent\small 
{\sc Figure 
\refstepcounter{figure}\thefigure}. 
The non-crossing annular permutation $\pi = (1, 2, 8)(3)(4, 6,
  7)(5)$. }$\hfill\hbox{}    
\end{figure}

\begin{notation}
Recall from \cite[Def. 3.5]{mn} that for integers $m, n \geq 1$ we let
$S_{NC}(m,n)$ be the non-crossing annular permutations. We shall
briefly recall the details. Let
$\gamma_{m,n} = (1, 2, 3, \dots, m)(m+1, \dots, m+n)$ be the
permutation in $S_{m+n}$ with two cycles. A permutation $\pi
\in S_{m+n}$ is \textit{non-crossing annular} if $\pi$ 
$\#(\pi) + \#(\pi^{-1}\gamma_{m,n}) = m + n$ and $\pi$
has at least one cycle that connects $\{1, 2, \dots, m\}$ to $\{m+1,
\dots, m+n\}$. Such cycles are called \textit{through
  cycles}. See Figure 6.
\end{notation}

\begin{remark}\label{remark:pair-hyperoctahedral}
Now let us recall a basic formula. Let $\sigma \in S_n$ be a
permutation and $A_1, \dots A_n$ be $N \times N$
matrices. We recall that $\Tr_\sigma (A_1, \dots, A_n)$ is
the product over the cycles of $\sigma$ of traces of
products of $A$'s. More precisely
\[
\Tr_\sigma(A_1, \dots, A_n) 
=
\mathop{\prod_{c \in \sigma}}_{c = (i_1, \dots, i_k)}
\Tr(A_{i_1}\cdots A_{i_k}).
\]
It is a standard result that
\[
\Tr_\sigma(A_1, \dots, A_n) =
\sum_{i_1, \dots i_n=1}^N 
a^{(1)}_{i_1i_{\sigma(1)}} a^{(2)}_{i_2i_{\sigma(2)}}
\cdots a^{(n)}_{i_ni_{\sigma(n)}}.
\]
We then let
\[
\tr_\sigma(A_1, \dots, A_n) = N^{-\#(\sigma)}
\Tr_\sigma(A_1, \dots, A_n).
\]
Let us recall next a formula from \cite[Lemma 5]{mp}. 
If $p \in \cP_2(\pm n)$ is a pairing of $[\pm n]$ then
there are $\sigma \in S_n$ and $\eta \in \bZ_2^n$ such that
\[
\mathop{\sum_{i_{\pm 1}, \dots, i_{\pm n} = 1}}_{\ker(i) \geq p}^N
a^{(1)}_{i_1i_{-1}} \cdots a^{(n)}_{i_ni_{-n}}
=
\Tr_\sigma\big(A^{(\eta_1)}, \dots, A^{(\eta_n)}\big)
\]
which are obtained as follows. According to Remark
\ref{remark:pair_decomposition} the cycles of $p \delta$
occur in pairs so we may write
\[
p\delta = c_1 c'_1 \cdots c_k c'_k
\]
where $c'_i = \delta c_i^{-1}\delta$. If $c_i = (j_1, \dots,
j_r)$ with $j_1, \dots, j_r \in [\pm n]$ we let $\tilde c_i
= (|j_1|, \dots, |j_r|)$ and $\eta_{j_i} = j_i/|j_i|$. For
example if $n = 4$ and $p = (1, 3)(-1, 2)(-2, -3)(4, -4)$
then $p\delta = (1, 2, -3)(4)(-1, 3, -2)(-4)$. So $\sigma
=(1,2,3)(4)$ and $\eta = (1,1,-1,1)$. Thus 
\[
\mathop{\sum_{i_{\pm 1}, i_{\pm 2}, i_{\pm 3}, i_{\pm 4} =1
}}_%
{\ker(i) \geq p}^N 
a^{(1)}_{i_1 i_{-1}} a^{(2)}_{i_2i_{-2}} a^{(3)}_{i_3i_{-3}} a^{(4)}_{i_4 i_{-4}}
= \Tr(A_1A_2A^\str_3)\Tr(A_4).
\]
Note that there is not a canonical choice of a
representative $\sigma$ because we have to choose one cycle
from each pair $c_i$ or $c'_i$. However because $\Tr(A_1
\cdots A_k) = \Tr(A_k^\str \cdots A_1^\str)$, the value of
$\Tr_\sigma\big(A_1^{(\eta_1)}, \dots, A_n^{(\eta_n)}\big)$
is independent of the choices.
\end{remark}

\begin{lemma}\label{lemma:3}\end{lemma}
Suppose that $X$ is the $N \times N$ \goe{} and $A_1, \dots,
A_n$ is a set of constant matrices. Then
\begin{equation}\label{eq:expansion_with_constants}
\E(\tr(XA_1\cdots XA_n)) =
\frac{N^{\#(\sigma) -(n/2+1)}}{2^{n/2}}
\mathop{\sum_{\pi \in \cP_2(n)}}_{\epsilon \in \bZ_2^n}
\tr_\sigma(A_1^{(\eta_1)}, \dots, A_n^{(\eta_n)})
\end{equation}
where $(\sigma, \eta)$ depends on the pair $(\pi, \epsilon)$
in the manner described in Remark
\ref{remark:pair-hyperoctahedral}.

\begin{proof}
We write $X = \frac{1}{\sqrt{2N}}(G + G^\str)$. We repeat
the calculation from Lemma \ref{lemma:1}, now with the $A$'s
inserted.

\begin{eqnarray*}\lefteqn{
\E(\Tr(XA_1 \cdots XA_n))} \\
& = &
(2N)^{-n/2}\sum_{\epsilon \in \bZ_2^n} 
\E(\Tr(G^{(\epsilon_1)}A_1
\cdots G^{(\epsilon_n)}A_n)) \\
& = &
(2N)^{-n/2}\sum_{\epsilon \in \bZ_2^n} 
\sum_{i_{\pm 1}, \dots, i_{\pm n}}
\E((G^{(\epsilon_1)})_{i_1i_{-1}}\cdots
(G^{(\epsilon_1)})_{i_ni_{-n}})\\
&& \qquad\mbox{}\times
a^{(1)}_{i_{-1}i_2} \cdots a^{(n)}_{i_{-n}i_1} \\
& = &
(2N)^{-n/2}\sum_{\epsilon \in \bZ_2^n}
\sum_{j_{\pm 1}, \dots, j_{\pm n}} 
\E(g_{j_1j_{-1}}\cdots
g_{j_nj_{-n}})\\
&& \qquad\mbox{}\times
a^{(1)}_{j_{-\epsilon(1)}j_{\epsilon(2)}} \cdots a^{(n)}_{j_{-\epsilon(n)}j_{\epsilon(1)}} \qquad\mbox{(letting\ } j= i \circ \epsilon)\\
& = &
(2N)^{-n/2}\sum_{\epsilon \in \bZ_2^n}
\sum_{j_{\pm 1}, \dots, j_{\pm n}} 
\sum_{\pi \in \cP_2(n)} \mathds{1}_{\ker(j) \geq \pi \delta \pi \delta}\\
&& \qquad\mbox{}\times
a^{(1)}_{j_{-\epsilon(1)}j_{\epsilon(2)}} \cdots a^{(n)}_{j_{-\epsilon(n)}j_{\epsilon(1)}} \\
&& \kern140pt\mbox{(then letting\ } i= j \circ \delta \gamma^{-1}\epsilon)\\
& = &
(2N)^{-n/2}\sum_{\pi,\epsilon}
\mathop{\sum_{i_{\pm 1}, \dots, i_{\pm n}}}_%
{\ker(i) \geq \delta \gamma^{-1}\epsilon \pi \delta \pi \delta \epsilon \gamma \delta} 
a^{(1)}_{i_1i_{-1}} \cdots a^{(n)}_{i_ni_{-n}}. \\
\end{eqnarray*}
Note that $\delta \gamma^{-1}\epsilon \pi \delta \pi \delta
\epsilon \gamma \delta$ is a pairing so that we can now
write the last term using Remark
\ref{remark:pair-hyperoctahedral}. Let $\rho = \epsilon \pi
\delta \pi \epsilon$. Then $\delta \gamma^{-1}\epsilon \pi
\delta \pi \delta \epsilon \gamma \delta = \delta
\gamma^{-1} \delta \rho \gamma \delta$. Hence $\delta
\gamma^{-1}\epsilon \pi \delta \pi \delta \epsilon \gamma
\delta \cdot \delta = \delta \gamma^{-1}\delta \cdot \rho
\cdot \gamma$. So by Remark
\ref{remark:pair-hyperoctahedral} there is a pair $(\sigma,
\eta) \in S_n \times \bZ_2^n$ such that
\[
\mathop{\sum_{i_{\pm 1}, \dots, i_{\pm n}}}_%
{\ker(i) \geq \delta \gamma^{-1}\epsilon \pi \delta \pi \delta \epsilon \gamma \delta} 
a^{(1)}_{i_1i_{-1}} \cdots a^{(n)}_{i_ni_{-n}}
= \Tr_\sigma(A_1^{(\eta_1)}, \dots, A_n^{(\eta_n)}).
\]
Hence
\[\E(\tr(XA_1\cdots XA_n)) =
\frac{N^{\#(\sigma) -(n/2+1)}}{2^{n/2}}
\mathop{\sum_{\pi \in \cP_2(n)}}_{\epsilon \in \bZ_2^n}
\tr_\sigma(A_1^{(\eta_1)}, \dots, A_n^{(\eta_n)})
\]
where $(\sigma, \eta)$ depends on the pair $(\pi, \epsilon)$
in the manner described in Remark
\ref{remark:pair-hyperoctahedral}. 
\end{proof}

In Remark \ref{remark:highest_order} we showed that
$\#(\epsilon \gamma \delta \gamma^{-1} \epsilon \vee \pi
\delta \pi \delta) \leq n/2 + 1$ with equality only if $\pi
\in NC_2(n)$ and $\epsilon_r = - \epsilon_s$ for all $(r, s)
\in \pi$. Moreover in this case $\rho = \epsilon \pi \delta
\pi \epsilon = \pi \delta \pi \delta$ so that
$\delta\gamma^{-1}\delta \rho \gamma = \pi \gamma \delta
\gamma^{-1}\pi \delta$ and hence $\sigma = \pi\gamma =
K(\pi)$ and $\eta \equiv 1$. Also for a given $\pi$ there
are $2^{n/2}$ choices of $\epsilon$ such that $\epsilon_r =
- \epsilon_s$ for all $(r, s) \in \pi$. Hence the highest
order, $O(1)$, term of $\E(\tr(XA_1\cdots XA_n))$ is
\[
\sum_{\pi \in NC_2(n)} 
\tr_{K(\pi)}(A_1, \dots, A_n).
\]
Under our assumption of the existence of an infinitesimal
limit (Eq. (\ref{eq:infintesimal_limit})) we have
\begin{multline*}
\lim_N N\Big( \sum_{\pi \in NC_2(n)} 
\tr_{K(\pi)}(A_1, \dots, A_n) - 
\sum_{\pi \in NC_2(n)} \phi_{K(\pi)}(a_1, \dots, a_n)\Big) \\
=
\sum_{\pi \in NC_2(n)}
\partial \phi_{K(\pi)}(a_1, \dots, a_n) \\
=
\sum_{\pi \in NC(n)} \kappa_\pi(x, \dots, x) \ 
\partial \phi_{K(\pi)}(a_1, \dots, a_n),
\end{multline*}
where the last equality holds because $x$ is a semi-circular
operator.

In Corollary \ref{cor:sub_leading_order} we showed that the
second highest order, $O(N^{-1})$, term in
Eq.~(\ref{eq:expansion_with_constants}) is when $\#(\epsilon
\gamma \delta \gamma^{-1} \epsilon \vee \pi \delta \pi
\delta) = n/2$ and this only occurs when $\rho \in
NC_2^\delta(n, -n)$. So this term is
\begin{equation}\label{eq:second_term}
2^{-n/2} \mathop{\sum_{\pi \in \cP_2(n)}}_{\epsilon \in \bZ_2^n}
\tr_\sigma(A_1^{(\eta_1)}, \dots, A_n^{(\eta_n)})
\end{equation}
where $(\pi, \epsilon)$ are such that $\#(\epsilon \gamma
\delta \gamma^{-1} \epsilon \vee \pi \delta \pi \delta) =
n/2$ and $(\sigma, \eta)$ depends on the pair $(\pi,
\epsilon)$ in the manner described in Remark
\ref{remark:pair-hyperoctahedral}. The element $\rho =
\epsilon \pi \delta \pi \epsilon$ produced from such a pair
is in $NC_2^\delta(n -n)$ and is independent of $\epsilon$
in the sense that for each pair $(r, s) \in \pi$, $\rho$
only depends on the product $\epsilon_r \epsilon_s$. There
are $2^{n/2}$ ways of choosing an $\epsilon$ for a fixed
assignment of signs $\epsilon_r \epsilon_s$ for each $(r, s)
\in \pi$. Moreover every $\rho \in NC_2^\delta(n, -n)$ can
be obtained from some pair $(\pi, \epsilon)$. To see this
start with a $\rho \in NC_2^\delta(n, -n)$ and for each pair
$(r, s) \in \rho$ let $(|r|, |s|)$ be a pair of
$\pi$. Because of the symmetry $\delta\rho\delta = \rho$
each $(|r|, |s|)$ will appear twice. Choose $\epsilon$ so
that for each $(r,s) \in \rho$ we have $\epsilon_r
\epsilon_s = -1$ if $(r, s)$ is not a through string of
$\rho$ and $\epsilon_r \epsilon_s = 1$ if $(r, s)$ is a
through string of $\rho$. Thus we may write
Eq.~(\ref{eq:second_term}) as
\[
\sum_{\rho \in NC_2^\delta(n, -n)}
\tr_\sigma(A_1^{(\eta_1)}, \dots, A_n^{(\eta_n)}).
\]
and as $N \rightarrow \infty$ we get
\[
\sum_{\rho \in NC_2^\delta(n, -n)}
\phi_\sigma(a_1^{(\eta_1)}, \dots, a_n^{(\eta_n)}).
\]
Putting these two terms together we get
\begin{multline*}
\phi'(xa_1 \cdots xa_n) \\
= \lim_N N\big( \E(\tr(XA_1 \cdots XA_n))
- \sum_{\pi \in NC_2(n)} \phi_{K(\pi)}(a_1, \dots, a_n) \big)\\
= \sum_{\pi \in NC_2(n)}
\partial \phi_{K(\pi)}(a_1, \dots, a_n)
+
\sum_{\rho \in NC_2^\delta(n, -n)}
\phi_\sigma(a_1^{(\eta_1)}, \dots, a_n^{(\eta_n)}).
\end{multline*}

\bigskip

\begin{notation}
Given $\rho \in NC_2^\delta(n,-n)$ we construct $(\sigma,
\eta)$ as in Remark \ref{remark:pair-hyperoctahedral} and we
denote $\phi_\sigma(a_1^{(\eta_1)}, \dots, a_n^{(\eta_n)})$
by $\phi_{K^\delta(\rho)}(a_1, \dots, a_n)$. The
justification for this notation is that the pair $(\sigma,
\eta)$ comes from the cycles of $\delta \gamma^{-1} \delta
\rho \gamma$ which can be thought of as a type $B$ Kreweras
complement. See Figure \ref{fig:delta_Kreweras_complement}. 
\end{notation}

\setbox1=\hbox{\includegraphics[scale=0.9]{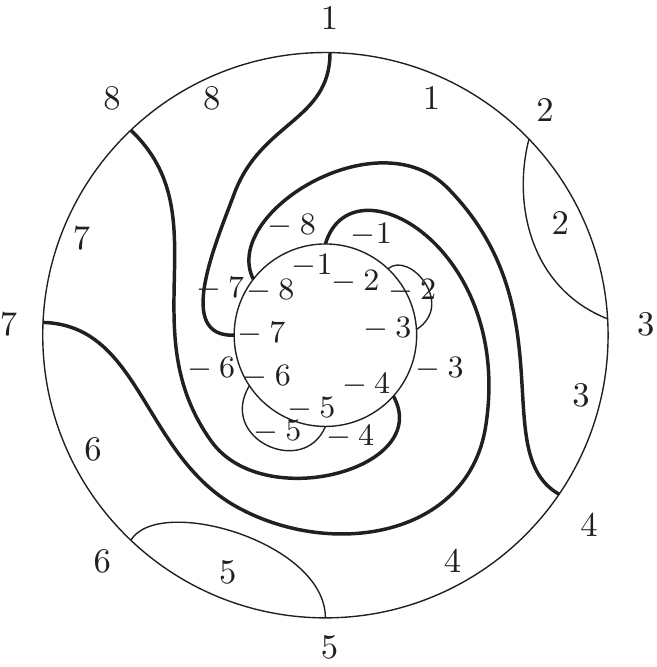}}

\begin{figure}[t]
\label{fig:delta_Kreweras_complement}
$\vcenter{\hsize=\wd1\box1}$
\hfill\hbox{}
$\vcenter{\hsize=180pt\raggedright\small
{\sc Figure} \refstepcounter{figure}\thefigure. 
If we let $\pi = (1,7)(2,3)(4,8)(5,6)\in \cP_2(8)$ and
$\epsilon = (1,1,\ab -1,1,1,-1,1,1)$, then $\rho = \epsilon
\pi \delta \pi \epsilon =
(1,-7)(2,3)\ab(4,-8)\ab(5,6)(-1,7)\ab
(-2,-3)(-4,8)(-5,-6)$. We compute $\delta \gamma^{-1} \delta
\rho \gamma = (1, 3, -7)(2) \ab (4, 6, -8) \ab (5)\ab (-1,
7, -3)\ab (-2)\ab (-4, 8,-6)\ab (-5)$. Then $\sigma =
(1,3,7)(2)(4,6,8)(5)$ and $\eta = (1,1,1,1,1,1,-1,-1)$. We
have $\phi_{K^\delta(\rho)}(a_1, a_2, a_3, a_4, a_5, a_6, a_7,
a_8)$ $= \phi(a_1a_3a_7^\str) \phi(a_2) \phi(a_4 a_6 a_8^\str)
\phi(a_5)$.}$
\end{figure}

\begin{theorem}\label{thm:lack_inf_free}
Suppose that $X$ is the $N \times N$ \goe{} and $A_1, \dots,
A_n$ is a set of constant matrices such that the $A$'s have
a joint infinitesimal limit distribution and $A_1, \dots,
A_n, A_1^\str, \dots, A_n^\str$ also have a joint limit
distribution. Then
\[
\lim_N N( \E(\tr(XA_1 \cdots XA_n)) -
\sum_{\pi \in NC(n)} \kappa_\pi(x, \dots, x)
\phi_{K(\pi)}(a_1, \dots, a_n))
\]
\[
=
\sum_{\pi \in NC(n)} \kappa_\pi(x, x, \dots, x)
\partial\phi_{K(\pi)}(a_{i_1}, a_{i_2}, \dots a_{i_n})
\]
\[
\mbox{} +
\sum_{\rho \in NC_2^\delta(n, -n)}
\kappa_\rho(x, \dots, x)
\phi_{K^\delta(\rho)}(a_1, \dots, a_n).
\]
\end{theorem}

\section{concluding remarks}\label{sec:concluding_remarks}
\begin{remark}\label{remark:possible_result}
By writing the second term as a sum over $NC_2^\delta(n, -n)$ we do not need to have $\partial\kappa_\rho$ as in equation (\ref{eq:inf_free_eqn}) on p. \pageref{eq:inf_free_eqn}. 
However, since there is a bijective map from $NCC_2(n)$ to
$NC_2^\delta(n,\ab -n)$ (see the paragraph above and Remark
\ref{remark:bijection}) and $\partial\kappa_\pi(x, \dots, x)
\not= 0$ only for elements of $NCC_2(n)$ (\textit{c.f.}
Eq.~(\ref{eq:inf_cum_part_goe}) on p. \pageref{eq:inf_cum_part_goe}), there should be way of
writing the second term above as a sum over $NC(n)$. This
would make it closer to the equation (\ref{eq:inf_free_eqn}) 
 for infinitesimal
freeness.
\end{remark} 

\begin{remark}
As we have seen, the fact that the genus expansion for the
complex Wishart means that the infinitesimal cumulants are
fairly simple: $\kappa'_n = c'$ for all $n$. In a follow-up
paper we shall compute the infinitesimal cumulants of a real
Wishart matrix. We get the $c'$ term as above plus a
polynomial in $c$.
\end{remark}

\thebottomline
\begin{thebibliography}{XX}


\bibitem{amsw} C.~Armstrong, J.~A.~Mingo, R.~Speicher,
  J.~C.~H.~Wilson, The Non-Commutative Cycle Lemma,
  \textit{J. of Comb. Thry.}, Ser. A \textbf{117}
  (2010). 1158-1166.

\bibitem{bai_s} Z.~D.~Bai and J.~Silverstein,
  \textit{Analysis of Large Dimensional Random Matrices},
  Springer Series in Statistics, 2009.

\bibitem{bs} S.~Belinschi and D.~Shlyakhtenko, Free
  Probability of Type B: Analytic Interpretation and
  Applications, \textit{Amer. J. Math.} \textbf{134} (2012),
  193-234.

\bibitem{bgn} P.~Biane, F.~Goodman, and A.~Nica,
  Non-crossing Cumulants of Type B,
  \textit{Trans. Amer. Math. Soc.} \textbf{355} (2003),
  2263-2303.

\bibitem{cc} M. Capitaine, M. Casalis. 
Asymptotic freeness by generalized moments for Gaussian and Wishart matrices.
Application to Beta random matrices, \textit{Indiana Univ. Math. J.}
\textbf{53} 2004, 397-431. 

\bibitem{cmss} B. Collins, J. A. Mingo,
  P. \'Sniady, and R. Speicher, \newblock{Second Order
    Freeness and Fluctuations of Random Matrices:
    III. Higher Order Freeness and Free Cumulants},
  \newblock{\em Documenta Math.}, \textbf{12} (2007), 1-70.

\bibitem{de} I.~Dumitriu and A.~Edelman, Global Spectrum
  fluctuations for the $\beta$-Hermite and $\beta$-Laguerre
  ensembles via matrix models, \textit{J.~Math. Phy.},
  \textbf{47}, 063302 (2006).


\bibitem{em} N.~Enriquez and L.~M\'enard, Asymptotic
  Expansion of the Expected Spectral Measure of Wigner
  Matrices, \textit{Electron. Commun. Probab.} \textbf{21}
  (2016), no. 58, 1-11.

\bibitem{f} M.~F\'evrier, Higher order infinitesimal
  freeness. \textit{Indiana Univ. Math. J.} \textbf{61}
  (2012), 249-295.

\bibitem{fn} M.~F\'evrier and A.~Nica, Infinitesimal
  non-crossing cumulants and free probability of type B,
  \textit{J.~Funct.~Anal.} \textbf{258} (2010), 2983-3023.

\bibitem{gj} I.~Goulden and D.~Jackson, Maps in Locally Orientable
Surfaces and Integrals over Real Symmetric Surfaces,
\textit{Can. J. Math.}, \textbf{49} (1997), 865-882.

\bibitem{krew} G.~Kreweras, Sur les partitions non
  crois\'ees d'un cycle, \textit{Discrete Math.} \textbf{1}
  (1972), 333-350.

\bibitem{kj} K. Johansson, On fluctuations of eigenvalues of
  random Hermitian matrices, \textit{Duke Math. J.}
  \textbf{91} (1998), 151-204.

\bibitem{kms} T. Kusalik, J. A. Mingo, and R. Speicher,
  Orthogonal polynomials and fluctuations of random
  matrices, {\em J. Reine Angew.  Math.}, {\bf 604} (2007),
  1 - 46.

\bibitem{l} M.~Ledoux, A recursion formula for the moments
  of the Gaussian orthogonal ensemble, \textit{Annales de
    l'I.~H.~P.--Probabilit\'es et Statistiques},
  \textbf{45} (2009), 754-769.

\bibitem{mn} J.~A.~Mingo and A.~Nica, Annular Noncrossing
  Permutations and Partitions, and Second Order Asymptotics
  for Random Matrices, \textit{Int. Math. Res. Not.} 2004, 
  no. 28, 1413-1460.

\bibitem{mp} J.~A.~Mingo and M.~Popa, Real second order
  freeness and Haar orthogonal matrices,
  \textit{J. Math. Phy.} \textbf{54} (2013), 051701, 1-35.


\bibitem{ms} J.~A.~Mingo and R.~Speicher, \newblock {Second
  Order Freeness and Fluctuations of Random Matrices: I.
  Gaussian and Wishart matrices and Cyclic Fock spaces},
  \newblock{\em J. Funct. Anal.}, {\bf 235}, (2006),
  226-270.

\bibitem{ms2017} J.~A.~Mingo and R.~Speicher, \textit{Free 
Probability and Random Matrices}, Fields Institute Communications
\textbf{35}, Springer Nature, 2017.

\bibitem{mst} J.~A.~Mingo, E.~Tan, and R.~Speicher, Second
  Order Cumulants of Products,
  \textit{Trans. Amer. Math. Soc.} \textbf{361} (2009),
  4571-4781.
  
\bibitem{ns} A. Nica and R. Speicher, \textit{Lectures on
  the Combinatorics of Free Probability}, Cambridge
  Univ. Press, 2006.

\bibitem{r} C. E. I. Redelmeier, Real second-order
  freeness and the asymptotic real second-order freeness of
  several real matrix ensembles,
  \textit{Int. Math. Res. Not.}  \textbf{2014}, no. 12,
  pp. 3353-3395.  

\bibitem{s} D.~Shlyakhtenko, Free Probability of Type B and
  Asymptotics of Finite Rank Perturbations of Random
  Matrices, \textit{Indiana Univ.~Math.~J.} \textbf{67} 
  (2018), 971-991.


\bibitem{voi} D.~Voiculescu, Limit laws for random
  matrices and free products, \textit{Invent. Math.}
  \textbf{104} (1991), 201--220.
  
  
\bibitem{vdn} D.-V.~Voiculescu, K.~Dykema, A.~Nica,
  \textit{Free Random Variables}, CRM Monograph Series,
  \textbf{1}, 1992.


\end{thebibliography}
\end{document}